\documentclass[12pt]{amsart}
\usepackage{bbding}
\usepackage{amssymb}
\usepackage{amscd}
\usepackage{amssymb, amsthm, hyperref }
\usepackage{enumerate, amsfonts, latexsym,epsfig, color, lpic}
\usepackage{graphicx}
\input xy
\xyoption{all}

\newtheorem {theorem}{Theorem} [section]
\newtheorem {lemma} [theorem] {Lemma}
\newtheorem {proposition} [theorem] {Proposition}
\newtheorem {example} [theorem] {Example}

\newtheorem {definition} [theorem] {Definition}

\newtheorem {remark} [theorem] {Remark}

\newtheorem {convention} [theorem] {Convention}
\newtheorem{Claim}{Claim}

\begin{document}

\title[Equation Problem over central extensions of hyperbolic groups]
{\bf Equation Problem over central extensions of hyperbolic groups}

\author[Hao Liang]{Hao Liang}

\address{Hao Liang\\
MSCS UIC 322 SEO, \textsc{M/C} 249\\
851 S. Morgan St.\\
Chicago, IL 60607-7045, USA} \email{hliang8@uic.edu}

\begin{abstract}
The Equation Problem in finitely presented groups asks if there
exists an algorithm which determines in finite amount of time
whether any given equation system has a solution or not. We show
that the Equation Problem in central extensions of hyperbolic groups
is solvable.
\end{abstract}

\maketitle

\vspace{4mm}

\section{\large Introduction}

Let $G$  be a finitely presented group and let $C$ be a generating
set for $G$. Denote a set of variables by $U$.  An {\em equation
system} is a finite collection of equations $w_{i}=1$ where
$w_{i}\in (U\cup C)^{*}$. Let $\mathcal{E}=\{w_{i}=1\mid i=1\cdots
n\}$ be an equation system in $G$. A {\em solution} of $\mathcal{E}$
in $G$ is a map $f: U\rightarrow G$ such that the induced monoid
homomorphism (sending each $c\in C$ to itself) $\bar f:(U\cup
C)^{*}\rightarrow G$ maps $w_{i}$ to $1$ for $1\leq i\leq n$.

Let $\mathcal{C}$ be a class of finitely presented groups. The {\em
Equation Problem} in $\mathcal{C}$ asks the following question: Is
there an algorithm which takes as input a presentation of  a group
$G\in\mathcal{C}$ and a finite system of equations with constants in
$G$ and which decides whether there exists a solution or not? When
the answer is positive, we say that the Equation Problem in
$\mathcal{C}$ is solvable.

The Equation problem is a vast generalization of the word problem,
the conjugacy problem and the simultaneous conjugacy problem. A
direct consequence is the existence of groups with unsolvable
equation problem. In fact  there exists a free 3-step nilpotent
group of rank 2 with unsolvable equations problem (\cite{Tru}, see
also \cite{Rom}). Therefore the Equation Problem is strictly harder
than the word problem and the conjugacy problem since these two
problems are solvable in finitely generated nilpotent groups.

The most famous breakthrough in solving equations over groups is the
solution of the Equation Problem in free groups by Makanin
\cite{Mak}. In \cite{DGu} Dahmani and Guirardel and independently in
\cite{LS} Lohrey and Senizergues give an algorithm for solving
equations and inequations with rational constraints in virtually
free groups. One of the most successful methods of solving the
Equation Problem for groups is to reduce it to the Equation Problem
over a (virtually) free group. Rips and Sela \cite{RS} reduce the
Equation Problem in torsion free hyperbolic groups to the Equation
Problem in free groups. Dahmani and Guirardel \cite{DGu} reduce the
problem in hyperbolic groups (possibly with torsion) to solving
equations with rational constraints (See Def \ref{rational
constraints}) over virtually free groups. Diekert and Muscholl
\cite{DM} reduce the Equation Problem in right-angled Artin groups
to free groups.

Besides being a natural algorithmic problem in groups, the equation
problem has the following interpretation in terms of homomorphisms
between groups: Let $K$, $G$ be finitely presented groups and
$\langle x_{1},\cdots, x_{n} \mid r_{1},\cdots, r_{i}\rangle$ be a
presentation of $K$. Suppose one's favorite elements in $K$ are
$w_{1},\cdots,w_{l}$. Here $w_{i}$ are words in $x_{i}$. Let
$c_{1},\cdots,c_{l}$ be elements of $G$. Here is a natural question
one might want to ask: Is there a homomorphism from $K$ to $G$
sending $w_{i}$ to $c_{i}$? The question is equivalent to asking
whether the equation system $E = \{r_{i}=1 ,w_{j}=c_{j} \mid 1 \leq
i\leq m; 1\leq j\leq l\}$ has a solution in $G$. Under this
interpretation $x_{i}$ become variables and note that $r_{i}$ are
words in $x_{i}$. Hence a positive solution to the Equation Problem
for $G$ gives a way to determine algorithmically whether particular
kind of homomorphisms to $G$ exist.

In this paper, we consider the Equation Problem in central
extensions of hyperbolic groups. We prove the following theorem.

\begin{theorem}\label{Main}
The Equation Problem is solvable in central extensions of hyperbolic groups.
\end{theorem}

A source of interesting examples of central extensions of hyperbolic
groups is 3-manifold theory. The fundamental groups of a large and
interesting class of Seifert fibered spaces are central extensions
of hyperbolic groups. All these Seifert fibered spaces are
aspherical and the fundamental group is a complete topological
invariant. Hence information about maps between fundamental groups
of these spaces gives essentially all information up to homotopy
about maps between these spaces. As pointed out above the Equation
Problem is an important tool for studying maps between groups. We
believe there will be applications of Theorem \ref{Main} that lead
to new understanding of maps between Seifert fibered spaces.

The organization of the paper is as follow. In Section \ref{P} we recall important definitions and key results from
\cite{DGu} and \cite{NR2}, which are the main tools we use in this paper, and describe the idea of the proof of Theorem \ref{Main}. In Section \ref{automata}, we
construct the finite state automata that we use in the proof of Theorem \ref{Main}. In Section \ref{proof}, we prove Theorem
\ref{Main}.

The author would like to thank his thesis advisor, Daniel Groves,
for his excellent guidance and for suggesting the problem. He would
also like to thank Benson Farb for helpful suggestions and his
interest in this work.

\section{\large Preliminaries}\label{P}

In the section we recall important definitions and key results from
\cite{DGu} and \cite{NR2}.

\subsection{Equation Problem in hyperbolic groups}
Let $\delta\textgreater 0$. A geodesic triangle in a metric space is
said to be {\em $\delta$-slim} if each of its sides is contained in
the $\delta$-neighbourhood of the union of the other two sides. A
geodesic space X is said to be {\em hyperbolic} there exists some
$\delta\textgreater 0$ if every triangle in X is $\delta$-slim. A
finitely generated group is {\em hyperbolic} if its Cayley graph is
hyperbolic.

Let $\Gamma$ be a hyperbolic group,  $X$ be a finite symmetric
generating set of $\Gamma$, $K_{\Gamma}$ be the Cayley graph of
$\Gamma$ with respect to $X$ and $\delta$ be a hyperbolic constant
for $K_{\Gamma}$. Consider the Rips complex of $K_{\Gamma}$, denoted
by $R_{50\delta}(K_{\Gamma})$, whose set of vertices is $\Gamma$ and
whose simplices are subsets of $\Gamma$ of diameter at most
$50\delta$ in $K_{\Gamma}$. Denote the 1-skeleton of the barycentric
subdivision of $R_{50\delta}(K_{\Gamma})$ by $\mathcal K$.

The action of $\Gamma$ on $K_{\Gamma}$ extends to an action on
$\mathcal K$. The quotient $\mathcal{K}/\Gamma$ is a finite graph
and can be given the structure of a finite graph of finite groups
(vertices and edges are decorated by stabilizers of their preimages
in $\mathcal{K}$). Hence $\pi_{1}(\mathcal{K}/\Gamma)$ is virtually
free. A finite presentation of $\pi_{1}(\mathcal{K}/\Gamma)$ can be
computed.

As in \cite{DGu}, let $V$ be the set of paths $v$ in $\mathcal K$
which start at the identity and end at a vertex of $K_{\Gamma}$ up
to homotopy relative to endpoints.  Let $\pi$ denote the natural
homomorphism from $V$ to $\Gamma$ sending each path to its endpoint.
We give $V$ a group structure by defining $vv'$ to be the homotopy
class of the concatenation $v\cdot(\pi(v)v')$ (where $\pi(v)v'$ is
the translate of $v'$ by $\pi(v)\in \Gamma$).

Dahmani and Guirardel prove the following \cite[Lemma 9.9]{DGu}

\begin{lemma}
The group $V$ is virtually free. More precisely, it is isomorphic to
the fundamental group of the finite graph of finite groups
$\mathcal{K}/\Gamma$. In particular, a presentation of $V$ is
computable from a presentation of $\Gamma$.
\end{lemma}

We need a slightly stronger statement about $V$.

\begin{lemma}\label{explicit path presentation}
A finite presentation $\langle Y|R\rangle$ of $V$ is computable.
Moreover one can compute each $y\in Y$ as an explicit path in
$\mathcal K$.
\end{lemma}

Note that a finite presentation of $V$ is computable by Lemma 2.1.
The point of Lemma 2.2 is that one can compute each $y\in Y$ as an
explicit path in $\mathcal K$.

\begin{proof}
First we note that any finite neighborhood of $\mathcal K$ can be
constructed since the word problem of $\Gamma$ is solvable. Let
$\theta:T\rightarrow\mathcal K$ be the universal cover of $\mathcal
K$. By \cite[proof of Lemma 9.9]{DGu} we have:
\begin{enumerate}
\item $V$ acts on $T$ by isometries;
\item $\theta$ is $\pi$-equivariant;
\item $T/V$ (given a graph of group structure as in \cite[page 54]{Serre})
is isomorphic to $\mathcal{K}/\Gamma$ as graph of groups. In
particular they are isomorphic as graphs.
\end{enumerate}

Starting with a vertex $t_{0}\in T$ which is mapped to
$1_{\Gamma}\in \mathcal K$ by $\theta$, one can construct any finite
neighborhood  of $t_{0}$ in $T$ and compute the map $u$ in that
neighborhood.  (This is because one can construct any finite
neighborhood of $1_{\Gamma}\in \mathcal{K}$ and one can construct
any finite ball of the universal cover of a locally finite graph.)

Since the quotient map from $\mathcal{K}$ to $\mathcal{K}/\Gamma$
can be computed in any finite neighborhood of $1_{\Gamma}$,  one can
compute the quotient map from $T$ to $T/V$ explicitly in any finite
neighborhood of $t_{0}$.

Given any $v\in V$ explicitly as a path in $\mathcal{K}$ and any
point $t$ in $T$ in a explicitly constructed finite neighborhood of
$t_{0}$,  one can compute $v\cdot t$ explicitly. (Here we assume
that the explicitly constructed finite neighborhood of $t_{0}$ also
contains $v\cdot t$. )

One can compute elements of the vertex groups explicitly as paths in
$\mathcal{K}$: vertex groups are stabilizers of vertices of $T$. One
can find those vertices explicitly. For any given element of $V$,
one can check whether it fixes any given vertex. Hence one can
compute all elements in the stabilizer of a given vertex by checking
each element of $V$ according to their lengths (in $\mathcal{K}$).
One knows when to stop checking because one knows the size of each
stabilizer.

One can also compute all the $\gamma_{y}$ (as in \cite{Serre})
explicitly because given any two vertices of $T$ in the same
$V$-orbit , one can find an explicit element of $V$ that takes one
to the other.

Since the presentation of $V$ given by $\pi_{1}(T/V)$ has elements
of the vertex groups of $(G,T/V)$ and $\gamma_{y}$'s as generators,
the proof is complete.
\end{proof}

We now recall how Dahmani and Guirardel reduce an equation system in
$\Gamma$ to finitely many equation systems in $V$.

By introducing new variables, one can find an equivalent triangular
equation system for any equation of length greater than three. (The
picture behind this is that one can cut a polygon into triangles.)
For equation of length two, say $x^{2}=1$, one can replace it by the
triangular equation with $x^{2}1=1$. Therefore it suffice to
consider only triangular equation systems.

Denote a set of variables (unknowns) by $U$ and a finite set of
words in $X$ by $C$. Let $\mathcal{F}=\{
\gamma_{i,1}\gamma_{i,2}\gamma_{i, 3}=1, i=1,\cdot\cdot\cdot, n\}$
be a triangular equation system over $\Gamma$, where
$\gamma_{i,j}\in U\cup C$.

Let $\mu_{0}=8$ and $\lambda_{0}=400\delta m_{0}$, where $m_{0}$ is
a bound on the cardinality of balls of radius $50\delta$ in
$K_{\Gamma}$. Consider $(\lambda_{1}, \mu_{1}) = (\lambda_{0},
\mu_{0}+2+2/\lambda_{0})$, so that any concatenation of a
$(\lambda_{0}, \mu_{0})$-quasi-geodesic with a path of length $1$ at
each extremity is a $(\lambda_{1}, \mu_{1})$-quasi-geodesic.

\begin{definition}\label{def(QG(V))}
$\mathcal{QG}(V)\subset V$ is the set of elements such that the
corresponding reduced path in $\mathcal{K}$ is $(\lambda_{1},
\mu_{1})$-quasi-geodesic.
\end{definition}

Denote by $V_{\leq l}$ the set of elements of $V$ whose
corresponding reduced path in $\mathcal{K}$ has length at most $l$.

Dahmani and Guirardel use Rips and Sela's canonical representatives
\cite{RS} to prove the following key proposition \cite[Proposition
9.10]{DGu}, which allows them to reduce any equation system in
$\Gamma$ to finitely many equation systems in $V$.

\begin{proposition}
[Dahmani-Guirardel]\label{Dahmani-Guirardel} Then there exists a
computable constant $\kappa_{1}$ (depending on $\Gamma$ and
$\mathcal{F}$) such that for any solution $(g_{u})\in \Gamma^{U}$ we
have the following:

For each $u\in U\cup \bar U$, there exists $\tilde g_{u}\in
\mathcal{QG}(V) $ with $\tilde g_{\bar u}=(\tilde g_{u})^{-1}$ and
$\pi(\tilde g_{u})=g_{u}$. For each
$\gamma_{i,1}\gamma_{i,2}\gamma_{i,3}=1$ in $\mathcal{F}$ and for
each $j\in \{1,2,3 \hspace{2mm}mod\hspace{2mm} 3\}$, there exists
$l_{i,j}\in \mathcal{QG}(V)$ and $c_{i,j}\in V_{\leq \kappa_{1}}$
such that:
\begin{enumerate}
\item $\tilde g_{\gamma_{i,j}}=l_{i,j}c_{i,j}(l_{i,j+1})^{-1}$ in $V$;
\item $\pi(c_{i,1}c_{i,2}c_{i,3})=1$ in $\Gamma$.
\end{enumerate}

Conversely, given any family of elements of $V$: $(\tilde
g_{u})_{u\in U\cup \bar U}$; $(l_{i,j})_{    1\leq i\leq n, 1\leq
j\leq 3}$ and $(c_{i,j})_{    1\leq i\leq n, 1\leq j\leq 3}$
satisfying $\tilde g_{\bar u}=(\tilde g_{u})^{-1}$, (1) and (2),
respectively, the family $g_{u}= \pi(\tilde g_{u})$ is a solution of
$\mathcal{F}$.
\end{proposition}

The $\tilde g_{u}$'s are called {\em Canonical Representatives}.

For each tuple $\bar c$ of $c_{i,j}\in V_{\leq \kappa_{1}}$
satisfying (2), define an equation system $\mathcal{F}(\bar c)$ in
$V$ as follows: Let $\{v_{i,j} \mid 1\leq i\leq n, 1\leq j\leq 3 \}$
be a set of variables such that $v_{i,j}$ and $v_{i',j'}$ are the
same variable if and only if $\gamma_{i,j}$ and $\gamma_{i',j'}$ are
the same. Let $\{p_{i,j}\mid 1\leq i\leq n, 1\leq j\leq 3 \}$ be a
set of distinct variables. We define
\begin{equation*}  \mathcal{F}(\bar c)=
\left\{ \begin{array}{rl} p_{i,1}c_{i,1}(p_{i,2})^{-1}=v_{i,1}
                                \\ p_{i,2}c_{i,2}(p_{i,3})^{-1}=v_{i,2}
                                \\ p_{i,3}c_{i,3}(p_{i,1})^{-1}=v_{i,3}
                         \end{array} \right.   \hspace{4mm}   1\leq i\leq n; 1\leq j\leq 3;  \end{equation*}

We call $\mathcal{F}(\bar c)$ {\em tripod equation system}
associated with $\bar c$.

Suppose that $\{\bar v_{i,j}, \bar p_{i,j}\}$ is a solution of
$\mathcal{F}(\bar c)$ in $V$. Let $\tilde v_{i,j}=\pi(\bar
v_{i,j})$. Then since $c_{i,j}$ satisfy Condition (2) in Proposition
\ref{Dahmani-Guirardel} we have that
\begin{equation*}
\tilde v_{i,1}\tilde v_{i,2}\tilde v_{i, 3}=1, \hspace{4mm}
i=1,\cdot\cdot\cdot, n.
\end{equation*}
With these equations we are almost ready to say that $\{\tilde
v_{i,j}\}$ is a solution of $\mathcal{F}$. The one extra thing we
need is that $\tilde v_{i,j}$ should equal $\gamma_{i,j}$ whenever
$\gamma_{i,j}$ is a constant in $\Gamma$. This is ensured by using
Equation system with rational constraints, which is defined below.

\begin{definition}\label{rational constraints}
A {\em rational subset} of a finitely presented group $G$ is the
image of a regular language of some generating set under the
canonical projection.

Let $C$ be a generating set for $G$. Denote a set of variables by
$U$. An {\em equation system with rational constraints} is an a
finite collection of equations $w_{i}=1$ where $w_{i}\in (U\cup
C)^{*}$ with finite set of pairs $(u, R_{u})$, where $u\in U$ and
$R_{u}$ is a rational subset of $G$.

Let $\mathcal{E}=\{w_{i}=1\mid i=1\cdots n\}\cup\{(u, R_{u})\mid
u\in U)\}$ be an equation system with rational constraints in $G$. A
{\em solution} of $\mathcal{E}$ in $G$ is a map $f: U\rightarrow G$
such that $f(u)\in R_{u}$ and the induced monoid homomorphism
(sending each $c\in C$ to itself) $\bar f:(U\cup C)^{*}\rightarrow
G$ maps $w_{i}$ to $1$ for $1\leq i\leq n$.
\end{definition}

\begin{convention}
We will use the notation $u\in R_{u}$ instead of $(u, R_{u})$ when
we write rational constraints in equation systems.
\end{convention}

We add rational constraints to the tripod equations system
$\mathcal{F}(\bar c)$ as follow: If $\gamma_{i,j}$ is a constant in
$\Gamma$, let $L_{\gamma_{i,j}}$ be the rational subset (of $V$)
$\{\tilde g\in \mathcal{QG}(V)\mid \pi(\tilde g)=\gamma_{i,j}\}$.
Otherwise let $L_{\gamma_{i,j}}=V$. We define:
\begin{equation*}  \mathcal{F}_{R}(\bar c)=
\left\{ \begin{array}{rl} p_{i,1}c_{i,1}(p_{i,2})^{-1}=v_{i,1}
                                \\ p_{i,2}c_{i,2}(p_{i,3})^{-1}=v_{i,2}
                                \\ p_{i,3}c_{i,3}(p_{i,1})^{-1}=v_{i,3}
                                \\v_{i,j}\in L_{\gamma_{i,j}}
                         \end{array} \right.   \hspace{4mm}   1\leq i\leq n; 1\leq j\leq 3;  \end{equation*}

Suppose $\{\bar v_{i,j}, \bar p_{i,j}\}$ to denote a solution of
$\mathcal{F}_{R}(\bar c)$ in $V$. When $\gamma_{i,j}$ is a constant,
the rational constraint $v_{i,j}\in L_{\gamma_{i,j}}$ implies that
$\tilde v_{i,j}=\pi(\bar v_{i,j})=\gamma$. Therefore $\{\tilde
v_{i,j}\}$ is a solution of $\mathcal{F}$.

On the other hand, by Proposition \ref{Dahmani-Guirardel} if
$\mathcal{F}$ has a solution in $\Gamma$ then $\mathcal{F}_{R}(\bar
c)$ has a solution in $V$ for some tuple $\bar c$ of $c_{i,j}\in
V_{\leq \kappa_{1}}$ satisfying (2) in Proposition
\ref{Dahmani-Guirardel}. All such tuples $\bar c$ can be computed
since each $c_{i,j}$ has bounded length. Note that the length bound
$\kappa_{1}$ on $c_{i,j}$ given by Proposition
\ref{Dahmani-Guirardel} is determined by $\Gamma$ and $\mathcal{F}$.
Hence the set of tuples $\bar c$ of $c_{i,j}\in V_{\leq \kappa_{1}}$
satisfying (2) in Proposition \ref{Dahmani-Guirardel} is determined
by $\Gamma$ and $\mathcal{F}$. Since there are finitely many $\bar
c$, any equation system in any hyperbolic group can be reduced to
finitely many equation systems $\mathcal{F}_{R}(\bar c)$ in some
virtually free group, which are solvable by the next theorem
\cite[Theorem 3]{DGu}:
\begin{theorem}[Dahmani-Guirardel]\label{Virtually free}
There exists an algorithm which takes as input a presentation of a
virtually free group G, and a system of equations with constants in
G, together with a set of rational constraints, and which decides if
there exists a solution or not.
\end{theorem}

Rational constraints play a very important role in our proof of
Theorem \ref{Main} as we explain in the next subsection.

\subsection{Idea of the proof of Theorem 1.1}

Let $E$ be a central extension of a hyperbolic group by a finitely
generated abelian group. Given a finite presentation of $E$, by
\cite[Proposition 1.1]{BR} one can compute the finite presentations
of all terms in a short exact sequence
\begin{equation*}
1\rightarrow A\rightarrow E\rightarrow\Gamma\rightarrow 1
\end{equation*}
where $\Gamma$ is a hyperbolic group and $A$ is a finitely generated
abelian group. Denote the inclusion from $A$ to $E$ by $i$ and the
projection from $E$ to $\Gamma$ by $p$.

Let $\mathcal{E}=\{ e_{i,1}e_{i,2}e_{i, 3}=1, i=1,\cdot\cdot\cdot,
n\}$ be a triangular equation system in $E$. Denote by $\mathcal{F}$
the equation system $p(\mathcal{E})=\{p(e_{i,1})p(e_{i,2})p(e_{i,
3})=1, i=1,\cdot\cdot\cdot, n\}$ in $\Gamma$, where
$p(e_{i,j})=e_{i,j}$ if $e_{i,j}$ is a variable. Here is an attempt
to check whether $\mathcal{E}$ has a solution in $E$:

Use the algorithm in Theorem \ref{Virtually free} to solve all the
tripod equation systems with rational constraints
$\mathcal{F}_{R}(\bar c)$ associated to $\mathcal{F}$ and $\Gamma$.

If none of the $\mathcal{F}_{R}(\bar c)$ has a solution, then
$\mathcal{F}$ has no solution in $\Gamma$ and hence $\mathcal{E}$
has no solution in $E$.

Suppose some $\mathcal{F}_{R}(\bar c)$ has a solution. Let $\{\bar
v_{i,j},\bar p_{i,j}\}$ be a solution. Then $\{\pi(\bar v_{i,j})\}$
is a solution of $\mathcal{F}$ in $\Gamma$. Let $s:\Gamma\rightarrow
E$ be a section. In general, $\{s(\pi(\bar v_{i,j}))\}$ is not a
solution of $\mathcal{E}$. But we know that $s(\pi(\bar
v_{i,1}))s(\pi(\bar v_{i,2}))s(\pi(\bar v_{i,3}))\in A$ for all
$1\leq i\leq n$. We set up the following equation system in $A$:
\begin{equation*}
\mathcal{W}(\bar c, \bar v_{i,j},\bar
p_{i,j})=\{w_{i,1}w_{i,2}w_{i,3}=-s(\pi(\bar v_{i,1}))s(\pi(\bar
v_{i,2}))s(\pi(\bar v_{i,3}))\mid 1\leq i\leq n\},
\end{equation*}
where $w_{i,j}$ and $w_{i',j'}$ represent the same variable if
$e_{i,j}$ and $e_{i', j'}$ are the same variable in $\mathcal{E}$
and $w_{i,j}=e_{i,j}-s(p(\bar v_{i,j}))$ is a constant in $A$ when
$e_{i,j}$ is a constant in $E$. Suppose $\{\bar w_{i,j}\}$ is a
solution of $\mathcal{W}(\bar c, \bar v_{i,j},\bar p_{i,j})$. It is
easy to check that $\{\bar e_{i,j}=s(\pi(\bar v_{i,j}))\bar
w_{i,j}\}$ is a solution of $\mathcal{E}$. Linear algebra can be
used to solve $\mathcal{W}(\bar c, \bar v_{i,j},\bar p_{i,j})$ since
$A$ is finitely generated and abelian. We check all
$\mathcal{W}(\bar c, \bar v_{i,j},\bar p_{i,j})$ to see if they have
a solution in $A$. If at least one does, than $\mathcal{E}$ has a
solution.

We point out that if $\mathcal{E}$ has a solution, then the process
above will detect it. (To see this, suppose $\mathcal{E}$ has a
solution $\{\bar e_{i,j}\}$, then $\{p(\bar e_{i,j})\}$ is a
solution of $\mathcal{F}$. Hence there is some $\bar c$ such that
$\mathcal{F}(\bar c)$ has a solution $\{\bar v_{i,j},\bar p_{i,j}\}$
with $\pi(\bar v_{i,j})=p(\bar e_{i,j})$. It is easy to check that
$\{\bar w_{i,j}=\bar e_{i,j}-s(p(\bar v_{i,j}))\}$ is a solution of
$\mathcal{W}(\bar c, \bar v_{i,j},\bar p_{i,j})$.) Hence if the
above process terminates before any solution is found, then
$\mathcal{E}$ has no solution in $E$.

There is a obvious problem about solving $\mathcal{E}$ this way:
$\mathcal{F}=p(\mathcal{E})$ can have infinitely many solutions in
$\Gamma$ even if $\mathcal{E}$ has no solution. In this case at
least of the tripod equation system with rational constraints
$\mathcal{F}_{R}(\bar c)$ has infinitely many solutions. So there
are infinitely many $\mathcal{W}(\bar c, \bar v_{i,j},\bar p_{i,j})$
to check. But none of them has a solution and so the process of
checking will never terminate. Here is an explicit example of this
phenomenon:

\begin{example}
Let $S$ be the genus two surface and $T^{1}S$ be the unit tangent
bundle of $S$. The following short exact sequence defines
$\pi_{1}(T^{1}S)$ as a central extension of $\pi_{1}(S)$, which is
hyperbolic.
\begin{equation*}
1\rightarrow \mathbb{Z}\rightarrow <a, b, c, d, z\mid
[a,b][c,d]=z^{-2}, z\hspace{1mm} central    >\rightarrow <a, b, c,
d\mid [a, b][c,d]=1>\rightarrow 1
\end{equation*}
Let $x$ be a variable. The equation $[a,b][x, d]=1$ has no solution
in $\pi_{1}(T^{1}S)$ but its projection in $\pi_{1}(S)$ has
infinitely many solutions: $x=cd^{n}$ for all $n\in \mathbb{Z}$. To
see this, first solve $[a,b][x, d]=1$ in $\pi_{1}(S)$, where we have
$[a,b][x, d]=[a,b][c,d]$. By simply algebraic manipulation, we got
$x^{-1}cd=dx^{-1}c$. But in $\pi_{1}(S)$, the only elements that
commute with $d$ are of the form $d^{n}$ for $n\in\mathbb{Z}$. Hence
we have that $x=cd^{n}$ for all $n\in \mathbb{Z}$ are the only
solutions. Now any solution of $[a,b][x, d]=1$ in $\pi_{1}(T^{1}S)$
must project down to a solution of its projection in $\pi_{1}(S)$.
Hence if there is a solution, then it has the form $x=cd^{n}z^{m}$.
Plug it into $x$, we have $[a,b][c,d]=1$, which does not hold in
$\pi_{1}(T^{1}(S))$. Therefore $[a,b][x, d]=1$ has no solution in
$\pi_{1}(T^{1}(S))$.
\end{example}

To deal with the above problem, we use rational constraints.

For each $\mathcal{F}_{R}(\bar c)$ we will have finitely many ways
to add more rational constraints to it. We denote the resulting
equation systems with rational constraints by
$\mathcal{F}_{R}^{k}(\bar c)$. We will prove the following:
\begin{enumerate}
\item If $\{\bar v_{i,j}, \bar p_{i,j}\}$ is solution of $\mathcal{F}_{R}(\bar c)$ then it is a solution of $\mathcal{F}_{R}^{k}(\bar c)$ for some $k$.
\item If $\{\bar v_{i,j}, \bar p_{i,j}\}$ and $\{\bar v'_{i,j}, \bar p'_{i,j}\}$ are solutions of $\mathcal{F}_{R}^{k}(\bar c)$, then we have
\begin{equation*}
s(\pi(\bar v_{i,1}))s(\pi(\bar v_{i,2}))s(\pi(\bar
v_{i,3}))=s(\pi(\bar v'_{i,1}))s(\pi(\bar v'_{i,2}))s(\pi(\bar
v'_{i,3}))
\end{equation*}
and one can compute this value from the rational constraints in
$\mathcal{F}_{R}^{k}(\bar c)$. We denote the above value by $s(\bar
c, k)$.
\end{enumerate}

Now instead of solving possibly infinitely many equation systems
$\mathcal{W}(\bar c, \bar v_{i,j},\bar p_{i,j})$, we only need to
solve, for each $\bar c$ and $k$, the corresponding pair of equation
systems $\mathcal{F}_{R}^{k}(\bar c)$ and
\begin{equation*}
\mathcal{W}^{k}(\bar c)=\{w_{i,1}w_{i,2}w_{i,3}=-s(\bar c, k)\mid
1\leq i\leq n\}.
\end{equation*}
By (1), (2) we have that $\mathcal{E}$ has a solution in $E$ if and
only if $\mathcal{F}_{R}^{k}(\bar c)$ has a solution in $V$ and
$\mathcal{W}^{k}(\bar c)$ has a solution in $A$ for some $\bar c$
and $k$.

Since there are finitely many tuples $\bar c$ and for each $\bar c$
there are finitely many $k$, there are finitely many pairs of
$\mathcal{F}_{R}^{k}(\bar c)$ and $\mathcal{W}^{k}(\bar c)$. So any
equation system in $E$ can be reduced to finitely many equation
systems with rational constraints in $V$ and finitely many equation
systems in $A$. Therefore we know that the Equation Problem in
central extensions of hyperbolic groups is solvable.

To define the rational constraints in $\mathcal{F}_{R}^{k}(\bar c)$
we need tools from \cite{NR2}, which we recall in the next
subsection.

\subsection{Central extensions of hyperbolic groups}
In this subsection, we recall some definitions and facts from \cite{NR2}. We include the proof of some of these facts since they are not in \cite{NR2}.

Recall that $K_{\Gamma}$ is the Cayley graph of $\Gamma$ with
respect to $X$. Denote by $L$ the language (over $X$) of all
$(\lambda, \nu)$-quasi-geodesic words in $K_{\Gamma}$. Note that $L$
is regular (See \cite{Holt-Rees}). Let $N$ a finite-state automaton
accepting $L$. Then $L$ is an asynchronous biautomatic structure of
$\Gamma$ (See \cite{Gromov} and \cite{NR1}).

\begin{definition}[$L$-rational]
Let $\pi: X^{*}\rightarrow \Gamma$ be the canonical projection. A
subset $T$ of $\Gamma$ is called {\em $L$-rational} if the language
$\{w\in L\mid \pi(w)\in T\}$ is regular.
\end{definition}

Let $s:\Gamma\rightarrow E$ be a section and
$\sigma_{s}:\Gamma\times\Gamma\rightarrow A$ be the cocycle defined
by $s$.

\begin{definition}
The cocycle $\sigma_{s}$ is {\em $L$-regular} if
\begin{enumerate}
\item The sets $\sigma_{s}(g, \Gamma)$ and $\sigma_{s}(\Gamma, g)$ are
finite for each $g\in\Gamma$.
\item For each $h\in\Gamma$ and $a\in A$ the subset $\{g\in\Gamma\mid \sigma_{s}(g,h)=a\}$ is an
$L$-rational subset of $\Gamma$.
\end{enumerate}
\end{definition}

\begin{theorem}[Neumann-Reeves]\label{regularity of rho}
For any central extensions of hyperbolic groups $E$ defined by
$1\rightarrow A\rightarrow E\rightarrow\Gamma\rightarrow 1$, there
exists a section $\rho:\Gamma\rightarrow E$ such that
$\sigma_{\rho}$ is $L$-regular, where $L$ is any biautomatic
structure of $\Gamma$.
\end{theorem}

In \cite{NR2} the above theorem is proved for a specific biautomatic
structure of $L$.  (In \cite{NR2}, $L$ is the language of maximising
words. See \cite[Lemma 2.1]{NR2} for more detail. But then the above
theorem follows easily by applying \cite[Proposition 1.1]{NR1}.

\begin{remark}
Note that given a presentation of $E$, the $L$-rational structure of
$\{g\in\Gamma\mid \sigma_{\rho}(g,h)=a\}$ can be computed (i.e. a
finite state automaton which accepts all $L$-words representing
elements of $\{g\in\Gamma\mid \sigma_{\rho}(g,h)=a\}$  can be
constructed explicitly). One can see this by examining the proof of
the above theorem from \cite{NR2} and using the fact that any finite
balls of the Cayley graphs of $E$ and $\Gamma$ can be constructed.
\end{remark}

Let $\rho:\Gamma\rightarrow E$ be a section such that
$\sigma_{\rho}$ is $L$-regular. Hence we know that
$\{g\in\Gamma\mid\sigma_{\rho}(g, x)=a\}$ is $L$-rational for $x\in
X$ and $a\in A$. We also need the fact that
$\{g\in\Gamma\mid\sigma_{\rho}(x, g)=a\}$ is $L$-rational for $x\in
X$ and $a\in A$. Since \cite{NR2} doesn't include a proof, we give a
proof.

\begin{lemma}\label{right regularity }
$\{g\in\Gamma\mid\sigma_{\rho}(x, g)=a \}$ is $L$-rational for $x\in
X$ and $a\in A$.
\end{lemma}

\begin{proof}
Consider the finite subset $D=\{\rho(x)i(-\sigma_{\rho}(g, x))\mid
g\in \Gamma, x\in X\}^{\pm 1}$ of $E$ as in \cite[Proposition
2.2]{NR1}. If $w=x_{1}\cdots x_{n}\in X^{*}$ then there exists
$w'\in D^{*}$ whose initial segments have values $\rho(x_{1}),
\rho(x_{1}x_{2}),\cdots,\rho(x_{1}\cdots x_{n})$. Let $L'$ be the
language $L'=\{w'\mid w\in L\}$. Then by \cite[Proposition 2.2]{NR1}
$L'$ is a regular language. Let $Z$ be a generating set for $A$. It
is easy to see that $D\cup i(Z)$ is a generating set of $E$. Let
$K_{E}$ be the Cayley graph of $E$ with respect to $D\cup i(Z)$.

For $x\in X$ and $a\in A$ consider the set $F=\{(w_{1}, w_{2})\in
L'\times L'\mid \rho(x)w_{1}=w_{2}i(a)\}$. Since $L$ is a
biautomatic  structure on $\Gamma$, we can apply \cite[Proposition
2.2]{NR1} and see that there exist $K$ such that $w_{1}$ and $w_{2}$
$K$-fellow-travel in $K_{E}$ if $(w_{1}, w_{2})\in F$. Hence $F$ is
the language of a two-tape finite state automata. Therefore
$F_{1}=\{w_{1}\in L'\mid \exists w_{2}\in L',
\rho(x)w_{1}=w_{2}i(a)\}$, which is the projection of $F$ to the
first factor, is regular. Let $F'_{1}=\{v\in L\mid\sigma_{\rho}(x,
v)=a\}$. One can use the finite-state automaton accepting $F_{1}$ to
read words in $L$ and $F'_{1}$ is the language accepted by it.
Therefore $F'_{1}$ is regular, which is equivalent to
$\{g\in\Gamma\mid\sigma_{\rho}(x, g)=a \}$ being $L$-rational.
\end{proof}

\begin{convention}[s-coordinates]
For any section $s:\Gamma\rightarrow E$ and any $g\in \Gamma$, $a\in
A$, we denote by $(g,a)_{s}$ the element $s(g)i(a)$ in $E$ and call
$(g,a)$ the {\em $s$-coordinates} of $s(g)i(a)$. For simplicity we
omit the subscript when the section under consideration is clear.
\end{convention}

We will use a symmetric section to lift the solutions of
$\mathcal{E}$ in $\Gamma$ to $E$. The symmetric section $q$
associated to $\rho$ lands in an extension of $E$, which we now
define.

Since $A$ is finitely generated and abelian, it is isomorphic to
$\mathbb{Z}^{n}\oplus\mathbb{Z}_{d_{1}}\oplus\cdots\oplus\mathbb{Z}_{d_{m}}$.
We identify $A$ with
$\mathbb{Z}^{n}\oplus\mathbb{Z}_{d_{1}}\oplus\cdots\oplus\mathbb{Z}_{d_{m}}$
by fixing an isomorphism between them. Let
$A'=\mathbb{Z}^{n}\oplus\mathbb{Z}_{2d_{1}}\oplus\cdots\oplus\mathbb{Z}_{2d_{m}}$.

Let $\iota_{1}$ be the injective homomorphism from $A$ to $A'$
defined by
\begin{equation*}
\iota_{1}(a_{1},\cdots, a_{n}, b_{1},\cdots,b_{m})=(2a_{1},\cdots,
2a_{n}, 2b_{1},\cdots,2b_{m}).
\end{equation*}
This map determines a pushout extension $1\rightarrow A'\rightarrow
E'\rightarrow\Gamma\rightarrow 1$ in the following sense:  Let
$E'=\Gamma\times A'$ be the direct product of $\Gamma$ and $A'$ as
sets. The map from $A'$ to $E'$ is the inclusion from $A'$ to
$\{1\}\times A'$ and the map from $E'$ to $\Gamma$ is the projection
of $E'$ to the first factor. Make $E'$ into a group by defining
\begin{equation*}
(g_{1},a_{1})(g_{2},a_{2})=(g_{1}g_{2},
a_{1}+a_{2}+\iota_{1}(\sigma_{\rho}(g_{1}, g_{2}))).
\end{equation*}
Let $\iota_{2}$ denote the natural inclusion from $E$ to $E'$, which
maps $(g, a)$ to $(g, \iota_{1}(a))_{\rho}$.  Let $\rho'$ be the
section from $\Gamma$ to $E'$ defined by $\rho'=\iota_{2}\rho$.
Hence an element $(g, a')$ in $E'$ has $\rho'$-coordinates $(g, a')$
and we have $\sigma_{\rho'}(g,h)=\iota_{1}(\sigma_{\rho}(g,h))$ for
any $g,h\in \Gamma$.

\begin{convention}\label{iota 3}
Let $\iota_{3}$ be the map (not a homomorphism) from $A$ to $A'$
defined by
\begin{equation*}
\iota_{3}(a_{1},\cdots, a_{n}, b_{1},\cdots,b_{m})=(a_{1},\cdots,
a_{n}, b_{1},\cdots,b_{m}).
\end{equation*}
To simplify notation, in the rest of the paper, if
$\sigma_{\rho}(-,-)$ appears in the second component of the
$\rho'$-coordinates of an element in $E'$, it represents the element
$\iota_{3}(\sigma_{\rho}(-,-))$ in $A'$.
\end{convention}

\begin{definition}\label{symmetric section}
The {\em symmetric section} $q:\Gamma\rightarrow E'$ is defined by:
\begin{equation*}
q(g)=(g, -\sigma_{\rho}(g, g^{-1}))
\end{equation*}
in $\rho'$-coordinate of $E'$.
\end{definition}

\begin{lemma}
$q$ is symmetric, i.e., $q(g)q(g^{-1})=1$.
\end{lemma}

\begin{proof}
In $\rho'$-coordinates of $E'$ we have
\begin{eqnarray*}
q(g)q(g^{-1})&=&(g, -\sigma_{\rho}(g, g^{-1}))(g^{-1},
-\sigma_{\rho}(g^{-1}, g))\\&=&(1, -\sigma_{\rho}(g,
g^{-1})-\sigma_{\rho}(g^{-1},
g)+2\sigma_{\rho}(g,g^{-1}))\\&=&(1,\sigma_{\rho}(g,g^{-1})-\sigma_{\rho}(g^{-1},
g))
\end{eqnarray*}
One the other hand, in $\rho$-coordinate of $E$ we have
\begin{eqnarray*}
(g,0)(g^{-1},-\sigma_{\rho}(g,g^{-1}))=(1,-\sigma_{\rho}(g,g^{-1})+\sigma_{\rho}(g,g^{-1}))=(1,0).
\end{eqnarray*}
Hence $(g^{-1},-\sigma_{\rho}(g,g^{-1}))$ is the inverse of $(g,0)$.
So we have
\begin{eqnarray*}
(1,0)=(g^{-1},-\sigma_{\rho}(g,g^{-1}))(g,0)=(1,-\sigma_{\rho}(g,g^{-1})+\sigma_{\rho}(g^{-1},g)).
\end{eqnarray*}
Therefore we have
$-\sigma_{\rho}(g,g^{-1})+\sigma_{\rho}(g^{-1},g)=0$. So we know
\begin{eqnarray*}
q(g)q(g^{-1})=(1,\sigma_{\rho}(g,g^{-1})-\sigma_{\rho}(g^{-1},
g))=(1,0).
\end{eqnarray*}
\end{proof}

Let $\sigma_{q}$ be the cocycle corresponding to $q$.

\begin{lemma}\label{regularity of q}
$\sigma_{q}$ is $L$-regular.
\end{lemma}

\begin{proof}
Let $x\in X$ and $g\in \Gamma$. In the $\rho'$-coordinate we have
\begin{eqnarray*}
q(g)q(x)&=&(g, -\sigma_{\rho}(g, g^{-1}))(x, -\sigma_{\rho}(x,
x^{-1}))\\&=&(gx, -\sigma_{\rho}(g, g^{-1})-\sigma_{\rho}(x,
x^{-1})+2\sigma_{\rho}(g,x))
\end{eqnarray*}
On the other hand, we have $q(gx)=(gx, -\sigma_{\rho}(gx,
(gx)^{-1}))$. Therefore we have
\begin{eqnarray*}
\sigma_{q}(g, x)&=&\sigma_{\rho}(gx, (gx)^{-1})-\sigma_{\rho}(g,
g^{-1})-\sigma_{\rho}(x,
x^{-1})+2\sigma_{\rho}(g,x)\\&=&\sigma_{\rho}(g,
g^{-1})-\sigma_{\rho}(g, x)-\sigma_{\rho}(x^{-1},
g^{-1})+\sigma_{\rho}(x,x^{-1})\\& &-\sigma_{\rho}(g,
g^{-1})-\sigma_{\rho}(x,
x^{-1})+2\sigma_{\rho}(g,x)\\&=&\sigma_{\rho}(g,
x)-\sigma_{\rho}(x^{-1}, g^{-1}).
\end{eqnarray*}
From the above equation we know that $\sigma_{q}(\Gamma, x)$ is
finite since both $\sigma_{\rho}(\Gamma, x)$ and
$\sigma_{\rho}(x^{-1}, \Gamma)$ are finite. Similarly, one can show
that $\sigma_{q}(x, \Gamma)$ is finite. Let $a\in A'$. From the
above computation we have
\begin{eqnarray*}
&&\{g\in
\Gamma\mid\sigma_{q}(g,x)=a\}\\=&&\bigcup_{a_{1}-a_{2}=a}(\{g\in
\Gamma\mid\sigma_{\rho}(g, x)=a_{1}\}\cap\{g\in
\Gamma\mid\sigma_{\rho}(x^{-1},g^{-1})=a_{2}\})
\end{eqnarray*}
Since $\sigma_{\rho}$ is $L$-regular, the left side of the above
equation is a finite union and  $\{g\in \Gamma\mid\sigma_{\rho}(g,
x)=a_{1}\}$ is an $L$-rational set. By Lemma 3.1 and the fact that
the reverse of a regular language is a regular language, $\{g\in
\Gamma\mid\sigma_{\rho}(x^{-1},g^{-1})=a_{2}\}$ is a $L$-rational
set. Therefore $\{g\in \Gamma\mid\sigma_{q}(g, x)=a\}$ is
$L$-rational; so $\sigma_{q}$ is regular.
\end{proof}

\begin{lemma}
$\{g\in\Gamma\mid\sigma_{q}(x, g)=a\}$ is $L$-rational for $x\in X$
and $a\in A'$.
\end{lemma}

\begin{proof}
Reverse the roles of $g$ and $x$ in the proof of the last lemma.
\end{proof}

\section{\large Future Predicting Automata and Parity Predicting Automata}\label{automata}
In this section we define Future Predicting Automata and Parity
Predicting Automata. They allow us to define the rational
constraints we put on the variables of the Tripod equation systems
in $V$.

\subsection{Lifting rational constraints to V}\label{subsection of lifting map
}
Both Future Predicting Automata and Parity Predicting Automata
define rational subsets of $\Gamma$. We explain how to lift these
rational subsets to $V$ first. This allows us to determine and fix
the constants $\lambda$ and $\nu$ which define $L$.

Recall that $Y$ is a finite generating set of $V$. We define a morphism $\phi$ from $Y^{*}$ to
$X^{*}$ to lift rational constraints to $V$.

By Lemma \ref{explicit path presentation} we know each $y\in Y$ as
explicit path in $\mathcal{K}$. For each $y\in Y$ we choose and fix
a $K_{\Gamma}$-geodesic word $w_{y}$ representing $\pi(y)$. Let
$\phi: Y^{*}\rightarrow X^{*}$ be the monoid homomorphism induce by
the map from $Y$ to $X^{*}$ sending $y$ to $w_{y}$. For any regular
language $K\subset X^{*}$, it is a standard fact that
$\phi^{-1}(K)=\{w\in Y^{*}|\phi(w)\in K\}$ is a regular language
over $Y$. (See [L] Theorem 4.2.4 for a proof.) Let $\mathcal{QG}(V)$
be as in Definition \ref{def(QG(V))}.

\begin{lemma}\label{quasi-geodesic constants}
There exist $\lambda, \nu$ depending only on $\Gamma$ such that the
following is true: For any $v\in \mathcal{QG}(V)$, there exists
$w\in Y^{*}$ representing  $v$ such that $\phi(w)$ represents a
$(\lambda, \nu)$-quasi-geodesics in $K_{\Gamma}$.
\end{lemma}

\begin{proof}
For any adjacent vertices $s, t\in \mathcal{K}$, let $[s, t]$ denote
the unique edge between them. If $s=t$, let $[s, t]$ be $s$
considered as a constant path.

Suppose the length of $v$ (considered as a path in $\mathcal{K}$) is $n$. Let $s_{0},\cdot\cdot\cdot,s_{n} $
be the vertices along $v$. By the construction of $\mathcal{K}$,
there are two types of vertices in $\mathcal{K}$: vertices of
$K_{\Gamma}$ and barycenters. If $s_{i}$ is a barycenter of some
simplex of $R_{50\delta}(K_{\Gamma})$, let $s'_{i}$ be a vertex of
this simplex. If $s_{i}$ is a vertex of $K_{\Gamma}$, let
$s'_{i}=s_{i}$. Then $s'_{0},\cdot\cdot\cdot,s'_{n}$ determine a
path $v'$ in $\mathcal{K}$. Here the path between $s'_{i-1}$ and
$s'_{i}$ consists of the edges $[s'_{i-1}, s_{i-1}]$,
$[s_{i-1},s_{i}]$ and $[s_{i}, s'_{i}]$ and we denote this path by
$[s'_{i-1},s'_{i}]$. Note that $v'$ equals $v$ in $V$.

Since $s'_{i}$ is a vertex in $K_{\Gamma}$, it also represent an
element of $\Gamma$. Let $v_{i}=(s'_{i-1})^{-1}[s'_{i-1},s'_{i}]$.
Then $v_{i}$ are element of $V$ and $v=v'=v_{1}\cdots v_{n}$.

For each $v_{i}$, pick a $Y$-geodesic word $w_{i}$ representing
$v_{i}$. Let $w=w_{1}\cdots w_{n}$. Then $w$ is a $Y$-word
representing $v$. Since there are finitely many elements of $V$ of
length less than four, there is an upper bound $K_{1}$ on the length
of $w_{i}$. Here $K_{1}$ depends only on $\Gamma$. Hence $w$ has
length at most $nK_{1}$.

There is an upper bound $K_{2}$ depending only on $\Gamma$ on the
length of $\phi(y)$ for all $y\in Y$ since $Y$ is finite. Let $n'$
be the length of $\phi(w)$. Then we have
\begin{equation*}
n'\leq nK_{1}K_{2}
\end{equation*}

Let $d$ be the distance in $\mathcal{K}$ between the end points of
$v$. Let $d_{X}$ be the distance in $K_{\Gamma}$ between the end
points of $v$.

Since $\mathcal{K}$ is quasi-isometric to $K_{\Gamma}$, we have
\begin{equation*}
\frac{1}{K_{0}}d-C\leq d_{X}
\end{equation*}
for some $K_{0}$ and $C$ depending only on $\Gamma$.

We know that $v$ is a $(\lambda_{1}, \nu_{1})$-quasi-geodesic in
$\mathcal{K}$ since $v\in\mathcal{QG}(V)$, hence we have
\begin{equation*}
\frac{1}{\lambda_{1}}n-\nu_{1}\leq d.
\end{equation*}

From the three equations above, we have
\begin{equation*}
\frac{1}{K_{0}K_{1}K_{2}\lambda_{1}}n'-\nu_{1}-C\leq d_{X}.
\end{equation*}

Let $\lambda=K_{0}K_{1}K_{2}\lambda_{1}$ and $\nu=\nu_{1}+C$. Then
the above equation implies that $\phi(w)$ is a $(\lambda,
\nu)$-quasi-geodesics in $K_{\Gamma}$. Note that $(\lambda, \nu)$
depends only on $\Gamma$.
\end{proof}

\subsection{Future Predicting Automata}
We are now ready to define the Future Predicting Automata. We will
first define the Future Predicting Automaton, which accepts exactly
the language $L$. The Future Predicting Automaton has many accepting
states. The languages defined by each of these accepting states give
a partition of $L$ and these subsets of $L$ will be used to define
the rational constraints we need in the next section.

For the rest of the paper, we use $L$ to denote the regular language
of words over $X$ representing $(\lambda, \nu)$-quasi-geodesic in
$K_{\Gamma}$ , where $\lambda$ and $\nu$ are given by Lemma
\ref{quasi-geodesic constants}.

By Lemma \ref{regularity of q} $\{g\in
\Gamma\mid\sigma_{q}(g,x)=a\}$ is $L$-rational for $x\in X$, $a\in
A'$ and
\begin{equation*}
A_{x}=\{\sigma_{q}(g, x)\in A'\mid g\in \Gamma\}.
\end{equation*}
is finite. For each $x\in X$, $a\in A_{x}$ we choose and fix a
finite state automaton $M_{x, a}$ which accepts exactly the
$L$-words representing elements in $\{g\in \Gamma\mid\sigma_{q}(g,
x)=a\}$. Denote the set of states and the transition map of
$M_{x,a}$ by $S_{x,a}$ and $F_{x,a}$, respectively.

\begin{definition}[FPA]\label{FPA}
The {\em Future Predicting Automaton}, denoted by $M$, is defined as
follows:

{\bf States}:  $S=\prod_{x\in X, a\in A_{x}}S_{x,a}$.

{\bf Transition function}: $F:S\times X\rightarrow S$ is defined by
\begin{equation*}
F\Big((s_{x,a})_{x\in X, a\in A_{x}}, x'\Big)=\big(F_{x,a}(s_{x,a},
x')\big)_{x\in X, a\in A_{x}} .
\end{equation*}

{\bf Initial state}: $(I_{x,a})_{x\in X, a\in A_{x}}$, where
$I_{x,a}$ is the initial state of $M_{x,a}$.

{\bf Accepting state}: The set of accepting states $T$ consists of
states $(s_{x,a})_{x\in X, a\in A_{x}}\in S$ satisfying: for all
$x'\in X$ there exists a unique $a'\in A_{x'}$ such that $s_{x',a'}$
is an accepting state of $M_{x',a'}$.
\end{definition}

\begin{remark}
Note that $M$ has finitely many states since $X$, $A_{x}$ and
$S_{x,a}$ are finite for all $x\in X$, $a\in A_{x}$.
\end{remark}

\begin{lemma}\label{FPA is L}
The language accepted by the Future Predicting Automaton is $L$.
\end{lemma}

\begin{proof}
Suppose $w$ is accepted by $M$. Then $w$ ends up in a state in $T$.
By the definition of $T$, for any $x\in X$, there exists a unique
$a\in A_{x}$, such that $s_{x,a}$ is an accepting state of
$M_{x,a}$. This together with the definition of $M$ implies that $w$
is accepted by $M_{x,a}$. Then by the definition of $M_{x,a}$, we
know that $w$ is an $L$-word representing an element in $\{g\in
\Gamma\mid\sigma_{q}(g, x)=a\}$. In particular, $w$ is in $L$.

Now suppose $w$ is in $L$. Use $M$ to read $w$. Suppose it ends at
the state $(s_{x,a})_{x\in X, a\in A_{x}}$. Note that $s_{x,a}$ is
an accepting state of $M_{x,a}$ if and only if we have
$\sigma_{q}(w,x)=a$, where $w$ is interpreted as the element of
$\Gamma$ it represents. Hence for each $x'\in X$, there is a unique
$a'\in A_{x'}$ such that $s_{x',a'}$ is an accepting state of
$M_{x',a'}$. Therefore $(s_{x,a})_{x\in X, a\in A_{x}}$ is in $T$.
So $w$ is accepted by $M$.
\end{proof}

\begin{definition}
Let $\bar s=(s_{x,a})\in T$. The {\em Future Predicting Automaton}
associated with $\bar s$, denoted by $M(\bar s)$, is the finite
state automaton having the same states, transition function and
initial state as the Future Predicting Automaton $M$, but $\bar s$
as the only accepting state.
\end{definition}

Let $L(\bar s)$ be the regular language over $X$ accepted by $M(\bar
s)$. The following fact is obvious from the definition and Lemma
\ref{FPA is L}.

\begin{lemma}\label{partition of L}
$\{L(\bar s)\mid \bar s\in T\}$ is a finite partition of $L$.
\end{lemma}

For any $\bar s\in T$, let $M_{\bar s}$ be the finite state
automaton which has the same states, transition map and accepting
states as $M$, but has $\bar s$ as the initial state. Note that
$M_{\bar s}$ is different from $M(\bar s)$.

\begin{definition}\label{compatible}
A word $v\in X^{*}$ is {\em compatible} with $\bar s\in T$ if $v$ is
accepted by $M_{\bar s}$.
\end{definition}

Note that $v$ is compatible with $\bar s$ simply means that for any
$w\in L(\bar s)$, the word $wv$ is in $L$ (or equivalently accepted
by $M$). The following fact is clear from Definition
\ref{compatible}.

\begin{lemma}\label{regularity of compatible}
For any $\bar s\in T$, the language of all words compatible with
$\bar s$ is regular.
\end{lemma}

\begin{convention}
We interpret any $w\in X^{*}$ as the element of $\Gamma$ represented
by the word $w$ when any cocycle is applied to $w$.
\end{convention}

The next lemma is the key property of the Future Predicting
Automata.

\begin{proposition}\label{key property of FPA}
Suppose that $w_{1},w_{2}\in L(\bar s)$. Then
$\sigma_{q}(w_{1},v)=\sigma_{q}(w_{2},v)$ provided $v\in X^{*}$ is
compatible with $\bar s\in T$.
 \end{proposition}

\begin{proof}
Let $v=x_{1}x_{2}\cdot\cdot\cdot x_{l}$.  We argue by induction on
$l$.

Suppose $l=1$. Since $w_{1},w_{2}\in L(\bar s)$, they both end at
the same state of $M(\bar s)$. Hence for $a\in A_{x_{1}}$, we have
that $w_{1}$ and $w_{2}$, when read by $M_{x_{1}, a}$, end up in the
same states. By the definition of FPA there exists a unique
$a_{1}\in A_{x_{1}}$ such that $M_{x_{1}, a_{1}}$ accepts both
$w_{1}$ and $w_{2}$. Therefore by the definition of $M_{x_{1},
a_{1}}$, we have $\sigma_{q}(w_{1}, x_{1})=\sigma_{q}(w_{2},
x_{1})=a_{1}$.

By the cocycle condition of $\sigma_{q}$, for $i=1,2$ we have
\begin{multline*}
\sigma_{q}(w_{i}, x_{1}x_{2}\cdot\cdot\cdot x_{l})=\sigma_{q}(w_{i},
x_{1}x_{2}\cdot\cdot\cdot
x_{l-1})+\sigma_{q}(w_{i}x_{1}x_{2}\cdot\cdot\cdot x_{l-1},
x_{l})-\sigma_{q}(x_{1}x_{2}\cdot\cdot\cdot x_{l-1}, x_{l})
\end{multline*}
By induction the first term does not depend on $i$. The third term
clearly does not depend on $i$. Showing that the second term does
not depend on $i$ will complete the proof.

Since $\bar s$ and $v$ are compatible, we have that
$w_{i}v=w_{i}x_{1}x_{2}\cdot\cdot\cdot x_{l}$ is in $L$. Therefore
we know that $w_{i}x_{1}x_{2}\cdot\cdot\cdot x_{l-1}$ is in $L$ for
$i=1,2$ since $L$ is closed under taking subword. In fact both
$w_{1}x_{1}x_{2}\cdot\cdot\cdot x_{l-1}$ and
$w_{2}x_{1}x_{2}\cdot\cdot\cdot x_{l-1}$ end up in the same
accepting state when read by $M$ since that is true for $w_{1}$ and
$w_{2}$. So there exists a unique $a_{l}\in A_{x_{l}}$ such that
$M_{x_{l}, a_{l}}$ accepts $w_{i}x_{1}x_{2}\cdot\cdot\cdot x_{l-1}$
for $i=1,2$. Hence $\sigma_{q}(w_{i}x_{1}x_{2}\cdot\cdot\cdot
x_{l-1}, x_{l})=a_{l}$, which does not depend on $i$.
\end{proof}

The above lemma explains the name ``Future Predicting Automaton"
since all we need to know to ``predict" the value of $\sigma_{q}(w,
v)$ is where $w$ ends when read by the Future Predicting Automaton.
By Proposition \ref{key property of FPA} the following definition
makes sense.

\begin{definition}
Let $\bar s\in T$ and $v\in X^{*}$. Suppose $\bar s$ and $v$ are
compatible. Define $\sigma_{q}(\bar s, v)$ to be $\sigma_{q}(w, v)$
for any $w\in L(\bar s)$.
\end{definition}

\subsection{Parity Predicting Automata}
Recall that $q:\Gamma\rightarrow E'$ is the symmetric section
(Definition \ref{symmetric section}). We are lifting solutions of
equation system in $\Gamma$ by $q$. We need to define appropriate
rational constraints on variables of the Tripod Equation systems so
that we can predict whether their lifts land in $E$ or not. The
Parity Predicting Automata define these rational constraints.

We first define the Left Future Predicting Automaton(LFPA) and the
Right Future Predicting Automaton(RFPA), which we will use to define
the Parity Predicting Automaton.

Let $\rho$ be the section given by Theorem \ref{regularity of rho}.
Hence we know that $\{g\in \Gamma\mid\sigma_{\rho}(g, x)=a\}$ is
$L$-rational for all $x\in X$ and $a\in A$ and
$A^{1}_{x}=\{\sigma_{\rho}(g, x)\mid g\in \Gamma\}$ is finite.

For each $x\in X$ and $a\in A^{1}_{x}$ choose and fix a finite state
automaton $M^{1}_{x, a}$ which accepts exactly the $L$-words
representing elements in $\{g\in\Gamma\mid\sigma_{\rho}(g,x)=a\}$.
Denote by $S^{1}_{x,a}$ and $F^{1}_{x,a}$ the set of states and the
transition function of $M^{1}_{x,a}$, respectively.

The following definition is almost the same as the definition of the
Future Predicting Automaton $M$, except that we are now considering
the section $\rho$ instead of $q$.

\begin{definition}
The {\em Left Future Predicting Automaton} $M_{1}$ over $X$ is
defined as follows:

{\bf States}: $S_{1}=\prod_{x\in X, a\in A^{1}_{x}} S^{1}_{x,a}$

{\bf Transition function}: $F_{1}:S_{1}\times X\rightarrow S_{1}$ is
defined by
\begin{equation*}
F_{1}\big((s^{1}_{x,a})_{x\in X, a\in A^{1}_{x}},
x'\big)=\big(F^{1}_{x,a}(s^{1}_{x,a}, x')\big)_{x\in X, a\in
A^{1}_{x}}.
\end{equation*}.

{\bf Initial state}: $(I^{1}_{x,a})_{x\in X, a\in A^{1}_{x}}$, where
$I^{1}_{x,a}$ is the initial states of $M^{1}_{x,a}$.

{\bf Accepting states}: The set of accepting states $T_{1}$ consists
of states $(s^{1}_{x,a})_{x\in X, a\in A^{1}_{x}}$ satisfying that
for all $x'\in X$ there exists a unique $a'\in A^{1}_{x'}$ such that
$s^{1}_{x',a'}$ is the accepting state of $M^{1}_{x',a'}$.
\end{definition}

One can prove the following lemma the same way as we prove Lemma
\ref{FPA is L}
\begin{lemma}
The language accepted by $M_{1}$ is $L$.
\end{lemma}

\begin{definition}
Suppose $\bar s_{1}=(s^{1}_{x,a})_{x\in X, a\in A^{1}_{x}}$ is an
accepting state of $M_{1}$. For any $x'\in X$, define
$\sigma_{\rho}(\bar s_{1}, x')=a'$, where $a'$ is the unique element
of $A^{1}_{x'}$ such that $s^{1}_{x',a'}$ is the accepting state of
$M^{1}_{x',a'}$.
\end{definition}

The next lemma follows directly from the above definition and the
definition of $M^{1}_{x,a}$.

\begin{lemma}\label{property of M_1}
Suppose $w\in X^{*}$ ends up in the state $\bar s_{1}\in T_{1}$ when
read by $M_{1}$. Then $\sigma_{\rho}(w, x)=\sigma_{\rho}(\bar
s_{1},x)$
\end{lemma}

By Lemma \ref{right regularity } the set of all the $L$-words
representing elements in $\{g\in \Gamma\mid\sigma_{\rho}(x, g)=a\}$
is a regular language. Hence its reverse is also a regular language.
Let $M^{2}_{x,a}$ be the finite state automaton accepting the
reverse language. Let $S^{2}_{x,a}$ and $F^{2}_{x,a}$ be the set of
states and the transition function of $M^{2}_{x,a}$, respectively.
Let $A^{2}_{x}=\{\sigma_{\rho}(x, g)\mid g\in \Gamma\}$. Note that
$A^{2}_{x}$ is finite by Theorem \ref{regularity of rho}.

\begin{definition}
The {\em Right Future Predicting Automaton} $M_{2}$ over $X$ is
defined as follows:

{\bf States}: $S_{2}=\prod_{x\in X, a\in A^{2}_{x}} S^{2}_{x,a}$.

{\bf Transition function}: $F_{2}:S_{2}\times X\rightarrow S_{2}$ is
defined as follow:
\begin{equation*}
F_{2}\big((s^{2}_{x,a})_{x\in X, a\in A^{2}_{x}},
x'\big)=\big(F^{2}_{x,a}(s^{2}_{x,a}, (x')^{-1})\big)_{x\in X, a\in
A^{2}_{x}}.
\end{equation*}

{\bf Initial state}: $(I^{2}_{x,a})_{x\in X, a\in A^{2}_{x}}$, where
$I^{2}_{x,a}$ is the initial states of $M^{2}_{x,a}$.

{\bf Accepting states}: The set of accepting states $T_{2}$ consists
of states $(s^{2}_{x,a})_{x\in X, a\in A^{2}_{x}}$ satisfying that
for all $x'\in X$ there exists a unique $a'\in A^{2}_{x'}$ such that
$s^{2}_{x',a'}$ is the accepting state of $M^{2}_{x',a'}$.
\end{definition}

One can prove the following lemma the similar way as we prove Lemma
\ref{FPA is L}

\begin{lemma}
The language accepted by $M_{2}$ is
\begin{equation*}
L^{-1}=\{w^{-1}=w_{n}^{-1}\cdots w_{1}^{-1}\mid w=w_{1}\cdots
w_{n}\in L\}.
\end{equation*}
\end{lemma}

\begin{definition}
Suppose $\bar s_{2}=(s^{2}_{x,a})_{x\in X, a\in A^{2}_{x}}$ is an
accepting state of $M_{2}$. Define $\sigma_{\rho}(x',\bar
s_{2})=a'$, where $a'$ is the unique element of $A^{1}_{x'}$ such
that $s^{1}_{x',a'}$ is the accepting state of $M^{1}_{x',a'}$.
\end{definition}

The next lemma follows directly from the above definition and the
definition of $M^{2}_{x,a}$.

\begin{lemma}\label{property of M_2}
Suppose $w\in X^{*}$ ends up in the state $\bar s_{2}\in T_{2}$ when
read by $M_{2}$. Then $\sigma_{\rho}(x, w^{-1})=\sigma_{\rho}(x,
\bar s_{2})$
\end{lemma}

We now define a homomorphism $Pa$, which makes what we mean by
``parity'' precise. Denote the natural project from $\mathbb{Z}$ to
$\mathbb{Z}_{2}$ by $P_{2}$. Recall that the finite generated
abelian group $A$ central in $E$ is identified with
$\mathbb{Z}^{n}\oplus\mathbb{Z}_{d_{1}}\oplus\cdots\oplus\mathbb{Z}_{d_{m}}$.
The homomorphism $Pa: A \rightarrow
\mathbb{Z}_{2}^{n}\oplus\mathbb{Z}_{d_{1}}\oplus\cdots\oplus\mathbb{Z}_{d_{m}}
$ is defined by
\begin{equation*}
Pa(a_{1},\cdots, a_{n}, b_{1},\cdots,b_{m})=(P_{2}(a_{1}),\cdots,
P_{2}(a_{n}), b_{1},\cdots,b_{m}).
\end{equation*}
The {\em parity} of an element $a\in A$ is define to be $Pa(a)$.

Let $a\in A$. Recall that $\iota_{3}$ and $\iota_{1}$ are defined in
Convention \ref{iota 3} and the paragraph before it, respectively.
In general $\iota_{3}(a)$ does not lie in $\iota_{1}(A)$. But once
we know the parity of $a$. We can use it to ``move'' $\iota_{3}(a)$
to something in $\iota_{1}(A)$. We now make this precise: Define the
map (not a homomorphism) $\iota_{4}:
\mathbb{Z}_{2}^{n}\oplus\mathbb{Z}_{d_{1}}\oplus\cdots\oplus\mathbb{Z}_{d_{m}}
\rightarrow
A'=\mathbb{Z}^{n}\oplus\mathbb{Z}_{2d_{1}}\oplus\cdots\oplus\mathbb{Z}_{2d_{m}}
$ by
\begin{equation*}
(a_{1},\cdots, a_{n}, b_{1},\cdots,b_{m})=(a_{1},\cdots,
a_{n}, b_{1},\cdots,b_{m})
\end{equation*}
Then the following fact is clear.

\begin{lemma}\label{fixing the one half}
For any $a\in A$, we have $\iota_{3}(a)+\iota_{4}(Pa(a))\in \iota_{1}(A)$.
\end{lemma}

Now we define the Parity Predicting Automaton.

\begin{definition}[PPA]
The {\em Parity Predicting Automaton}, denoted by $D$, is defined as
follows:

{\bf States}: $S_{D}=S_{1}\times S_{2}\times
(\mathbb{Z}_{2}^{n}\oplus\mathbb{Z}_{d_{1}}\oplus\cdots\oplus\mathbb{Z}_{d_{m}}
)\cup\{\emptyset\}$.

{\bf Transition functions}: $F_{D}:S_{D}\times X\rightarrow S_{D}$
is defined as follows:

If $\bar s_{1}\in T_{1}$ and $\bar s_{2}\in T_{2}$, then we define
\begin{eqnarray*}
&&F_{D}\Big((\bar s_{1} , \bar s_{2},  b), x\Big)
\\&=&\Big(F_{1}(\bar s_{1}, x), F_{2}(\bar s_{2},
x),  b'\Big),
\end{eqnarray*}
where
\begin{eqnarray*}
b'=b+Pa\Big(\sigma_{\rho}\big(x,
x^{-1}\big)-\sigma_{\rho}\big(\bar s_{1},
x\big)-\sigma_{\rho}\big(x^{-1}, \bar s_{2}\big)\Big).
\end{eqnarray*}

Otherwise we define
\begin{eqnarray*}
F_{D}\Big(\big(\bar s_{1}, \bar s_{2},  b\big),
x\Big)=\emptyset.
\end{eqnarray*}
For all $x\in X$, we define
\begin{equation*}
F_{D}(\emptyset, x)=\emptyset.
\end{equation*}

{\bf The initial state}: $(I_{1}, I_{2}, 0)$. Here $I_{1}$ and
$I_{2}$ are the initial states of $M_{1}$ and $M_{2}$ respectively.

{\bf The accepting states}: All states $(\bar s_{1},\bar s_{2},
b)$ so that $\bar s_{1}\in T_{1}$ and $\bar s_{2}\in T_{2}$.
\end{definition}

The following fact is clear from the above definition.

\begin{lemma}
The language accepted by $D$ is $L$.
\end{lemma}

Let $d\in
\mathbb{Z}_{2}^{n}\oplus\mathbb{Z}_{d_{1}}\oplus\cdots\oplus\mathbb{Z}_{d_{m}}$.

\begin{definition}
The {\em Parity Predicting Automaton} associated to $d$,
denoted by $D(d)$, is the same as the Parity Predicting
Automaton with the extra requirement that the third component of any
accepting state is $d$.
\end{definition}

Let $L(d)$ denote the regular language defined by $D(d)$.

\begin{lemma}
$\{L(d)\mid  d\in
\mathbb{Z}_{2}^{n}\oplus\mathbb{Z}_{d_{1}}\oplus\cdots\oplus\mathbb{Z}_{d_{m}}\}$
is a finite partition of $L$.
\end{lemma}

The following lemma is the key property of the Parity Predicting
Automata.

\begin{lemma}\label{key property of PPA}
Let $w\in L(d)$. Then $Pa\big(\sigma_{\rho}(w,w^{-1})\big)=
d$.
\end{lemma}

\begin{proof}
We proceed by induction on the length of $w$. The base case is when
$w$ has length $0$. In this case, $w$ is the identity. Hence
$\sigma_{\rho}(w,w^{-1})=0$. Since $D$'s initial state has $0$ as
its last component, the lemma is true in this case.

Suppose the lemma is true for words of length less then $l$. Let
$w=x_{1}\cdots x_{l}$. By the cocycle condition of
$\sigma_{\rho}$ we have
\begin{eqnarray*}
\sigma_{\rho}(w,w^{-1}) &=& \sigma_{\rho}(x_{1}\cdots x_{l}, x_{l}^{-1}\cdots x_{1}^{-1})
\\&=&  \sigma_{\rho}(x_{1}\cdots x_{l}, x_{l}^{-1})+\sigma_{\rho}(x_{1}\cdots x_{l-1}, x_{l-1}^{-1}\cdots x_{1}^{-1})\\& &-\sigma_{\rho}(x_{l}^{-1},x_{l-1}^{-1}\cdots x_{1}^{-1})
\\&=&\sigma_{\rho}(x_{1}\cdots x_{l-1},1)-\sigma_{\rho}(x_{1}\cdots x_{l-1},x_{l})+\sigma_{\rho}(x_{l},x_{l}^{-1})\\& &+\sigma_{\rho}(x_{1}\cdots x_{l-1}, x_{l-1}^{-1}\cdots x_{1}^{-1})-\sigma_{\rho}(x_{l}^{-1},x_{l-1}^{-1}\cdots x_{1}^{-1})
\\&=&\sigma_{\rho}(x_{1}\cdots x_{l-1}, x_{l-1}^{-1}\cdots x_{1}^{-1})+\sigma_{\rho}(x_{l},x_{l}^{-1})\\& &-\sigma_{\rho}(x_{1}\cdots x_{l-1},x_{l})-\sigma_{\rho}(x_{l}^{-1},x_{l-1}^{-1}\cdots x_{1}^{-1})
\end{eqnarray*}

Suppose $x_{1}\cdots x_{l-1}$ ends at a state $(\bar s_{1}, \bar
s_{2}, b_{l-1})$. Note that $\bar s_{1}\in T_{1}$ and $\bar
s_{2}\in T_{2}$ (otherwise $w$ wouldn't be accepted by $D$).  By the
induction hypothesis, $Pa\big(\sigma_{\rho}(x_{1}\cdots x_{l-1},
x_{l-1}^{-1}\cdots x_{1}^{-1})\big)=b_{l-1}$. Then by the
equation above we have
\begin{eqnarray*}
&&Pa\big(\sigma_{\rho}(w,w^{-1}) \big)
\\&=& Pa\big(\sigma_{\rho}(x_{1}\cdots x_{l-1}, x_{l-1}^{-1}\cdots x_{1}^{-1})\big)+
\\& &Pa\big(\sigma_{\rho}(x_{l},x_{l}^{-1})-\sigma_{\rho}(x_{1}\cdots x_{l-1},x_{l})-\sigma_{\rho}(x_{l}^{-1},x_{l-1}^{-1}\cdots x_{1}^{-1})\big)
\\&=&b_{l-1}+Pa\big(\sigma_{\rho}(x_{l},x_{l}^{-1})-\sigma_{\rho}(x_{1}-\cdots x_{l-1},x_{l})-\sigma_{\rho}(x_{l}^{-1},x_{l-1}^{-1}\cdots x_{1}^{-1})\big).
\end{eqnarray*}

On the other hand, by the definition of $D$, $x_{1}\cdots x_{l}$
ends at the state $\Big(F_{1}\big(\bar s_{1}, x_{l}\big),
F_{2}\big(\bar s_{2}, x_{l}^{-1}\big),  b\Big)$, where
\begin{eqnarray*}
 b= b_{l-1}+Pa\big(\sigma_{\rho}(x_{l},
x_{l}^{-1})-\sigma_{\rho}(\bar s_{1},
x_{l})-\sigma_{\rho}(x_{l}^{-1},\bar s_{2})\big)
\end{eqnarray*}
Therefore it is enough to show that
\begin{equation*}
\sigma_{\rho}(x_{1}\cdots x_{l-1},x_{l})=\sigma_{\rho}(\bar s_{1},
x_{l})
\end{equation*}
and
\begin{equation*}
\sigma_{\rho}(x_{l}^{-1},x_{l-1}^{-1}\cdots
x_{1}^{-1})=\sigma_{\rho}(x_{l}^{-1}, \bar s_{2}).
\end{equation*}

Since $x_{1}\cdots x_{l-1}$ ends at $(\bar s_{1}, \bar s_{2},
b_{l-1})$, we know that $x_{1}\cdots x_{l-1}$ ends at $\bar s_{1}$
when read by $M_{1}$ and it ends at $\bar s_{2}$ when read by
$M_{2}$. Hence by Lemma \ref{property of M_1} and Lemma
\ref{property of M_2}, the above two equations follow and the proof
of the lemma is completed.
\end{proof}

\section{\large Proof of Theorem \ref{Main}}\label{proof}

Let $U$ be a finite set of variables and $C\subset E$ be a finite
set of constants in $E$. Recall from Section 2 that it is enough to
consider triangular equation systems. Let $\mathcal{E}=\{
e_{i,1}e_{i,2}e_{i, 3}=1, i=1,\cdot\cdot\cdot, n\}$ be a triangular
equation system where $e_{i,j}\in U\cup C$.

Now we construct equation systems $\mathcal{V}_{t}$ over $V$ and
$\mathcal{W}_{t}$ over $A$ where $t$ runs over some finite set
$\Theta$. The size of $\Theta$ depends on $\mathcal{E}$ and
$\Gamma$. Then we show that $\mathcal{E}$ has a solution in $E$ if
and only if there is some $t\in\Theta$ such that $\mathcal{V}_{t}$
has a solution in $V$ and $\mathcal{W}_{t}$ has a solution in $A$.

\vspace{4mm}

First we describe the finite set $\Theta$ over which the subscript
$t$ of  $\mathcal{V}_{t}$ and  $\mathcal{W}_{t}$ runs. The index set
$\Theta$ consists of tuples $\big((c_{i,j}), (\bar s_{i,j}),
(b_{i,j}), (d_{i,j})\big)_{1\leq i\leq n, 1\leq j\leq 3}$
satisfying the following 4 conditions:

Recall that $V_{\leq l}$ is the set of elements of $V$ whose
corresponding reduced path in $\mathcal K$ has length at most $l$.
Let $\kappa_{1}$ be as in Proposition \ref{Dahmani-Guirardel}. There
exists $\kappa_{2}$ such that all elements of $V_{\leq \kappa_{1}}$
are represented by some $Y$-words of word length as most
$\kappa_{2}$. Recall that $\phi$ is the monoid homomorphism from $Y^{*}$ to
$X^{*}$ defined in the Section 3.1.

\vspace{2mm} {\bf Condition 1}: For each $1\leq i\leq n$, $1\leq
j\leq 3$,  $c_{i,j}$ is a $Y$-word of word length at most
$\kappa_{2}$ and for each $i$ we have $\pi(c_{i,1}c_{i,2}c_{i,3})=1$
in $\Gamma$.

\vspace{2mm} {\bf Condition 2}: For $1\leq i\leq n$, $1\leq j \leq
3$, $\bar s_{i,j}\in T$ is an accepting state of the Further
Predicting Automaton so that $\bar s_{i,j}$ and $\phi(c_{i,j})$ are
compatible.

\vspace{2mm} Recall that for any $\bar s\in T$, $M_{\bar s}$ is the
finite state automaton which has the same states, transition map and
accepting states as $M$, but has $\bar s$ as the initial state. Let
$\bar s'_{i,j}$ be where $\phi(c_{i,j})$ ends when read by $M_{\bar
s_{i,j}}$. Denote the language of words in $X$ that are compatible
with $\bar s'_{i,j}$ by $L(\bar s_{i,j},c_{i,j})$. Let $A(\bar
s_{i,j}, c_{i,j})=\{\sigma_{q}(\bar s'_{i,j}, w)\mid w\in L(\bar
s_{i,j},c_{i,j}) \}$.

\vspace{2mm} {\bf Condition 3}: For all $1\leq i\leq n$, $1\leq
j\leq 3$, $b_{i,j}\in A(\bar s_{i,j}, c_{i,j})$.

\vspace{2mm} {\bf Condition 4}: For all $1\leq i\leq n$, $1\leq
j\leq 3$,  $d_{i,j}\in
\mathbb{Z}_{2}^{n}\oplus\mathbb{Z}_{d_{1}}\oplus\cdots\oplus\mathbb{Z}_{d_{m}}
$.

\begin{lemma}\label{theta is finite}
$\Theta$ is finite.
\end{lemma}

\begin{proof}
Since the lengths of $c_{i,j}$'s are bounded, there are finitely
many tuples $(c_{i,j})$ satisfying Condition 1. We know that $T$ is
finite for it is the set of accepting states of a finite state
automaton (FPA). Hence there are finitely many tuples $(\bar
s_{i,j})$ satisfying Condition 2.

For each choice of $(c_{i,j})$ and $(\bar s_{i,j})$, note that
$A(\bar s_{i,j}, c_{i,j})=\{\sigma_{q}(u_{i,j}, w)\mid w\in L(\bar
s_{i,j},c_{i,j}) \}$ for any $u_{i,j}\in L(\bar s'_{i,j})$. Hence
$A(\bar s_{i,j}, c_{i,j})$ is finite by Lemma \ref{regularity of q}.
Therefore there are finitely many choices for $(b_{i,j})$. At last,
possibilities of $(d_{i,j})$ are bounded since
$\mathbb{Z}_{2}^{n}\oplus\mathbb{Z}_{d_{1}}\oplus\cdots\oplus\mathbb{Z}_{d_{m}}$
is finite.
\end{proof}

For each $t=\big((c_{i,j}), (\bar s_{i,j}), (b_{i,j}), (\bar
d_{i,j})\big)_{1\leq i\leq n, 1\leq j\leq 3}$ we have the following
setups:

Let $L(\bar s_{i,j})\subset X^{*}$ be the regular language
associated to $\bar s_{i,j}$.

Let $L(b_{i,j})=\{w\in L(\bar s_{i,j},c_{i,j})\mid\sigma_{q}(\bar s'_{i,j},
w)=b_{i,j}\}$.

\begin{lemma}
$L(b_{i,j})$ is regular.
\end{lemma}

\begin{proof}
Pick and fix $u_{i,j}\in L(\bar s'_{i,j})$. Note that $L(b_{i,j})= \{w\in L\mid\sigma_{q}(u_{i,j}, w)=b_{i,j}\}\cap
L(\bar s_{i,j},c_{i,j})$. By Lemma \ref{regularity of q} $\{w\in
L\mid\sigma_{q}(u_{i,j}, w)=b_{i,j}\}$ is regular and $L(\bar
s_{i,j},c_{i,j})$ is regular by Lemma \ref{regularity of
compatible}. Hence $L(b_{i,j})$ is a regular language.
\end{proof}

Let $L(d_{i,j})\subset X^{*}$ be the regular language
associated to $d_{i,j}$.

If $e_{i,j}$
is a constant in $E$, let $L(e_{i,j})$ denote the regular language
of all $L$-representatives of $p(e_{i,j})$; otherwise, let
$L(e_{i,j})=L$.

\begin{convention}
For any regular language $K$, we use $\phi^{-1}(K)$ to denote the
rational subset of $V$ defined by $\phi^{-1}(K)\subset Y^{*}$.
\end{convention}

We now define $\mathcal{V}_{t}$ as follows:

Let $\{v_{i,j} \mid 1\leq i\leq n, 1\leq j\leq 3 \}$ be a set of
variables such that $v_{i,j}$ and $v_{i',j'}$ are the same variable
if and only if $e_{i,j}$ and $e_{i',j'}$ are the same.  Let $P=\{p_{i,j} | 1\leq i\leq n, 1\leq j\leq 3 \}$
be another set of distinct variables.
\begin{equation*}  \mathcal{V}_{t}=
\left\{ \begin{array}{rl} p_{i,1}c_{i,1}(p_{i,2})^{-1}=v_{i,1}
                                \\ p_{i,2}c_{i,2}(p_{i,3})^{-1}=v_{i,2}
                                \\ p_{i,3}c_{i,3}(p_{i,1})^{-1}=v_{i,3}
                                \\p_{i,j}\in \phi^{-1}(L(\bar s_{i,j}))
                                \\p_{i,j+1}^{-1}\in \phi^{-1}(L(b_{i,j}))
                                \\v_{i,j}\in \phi^{-1}(L(d_{i,j}))
                                \\v_{i,j}\in \phi^{-1}(L(e_{i,j}))
                         \end{array} \right.   \hspace{4mm}   1\leq i\leq n; 1\leq j\leq 3; 3+1=1 \end{equation*}

\vspace{4mm} Let $a_{i,j}=\sigma_{q}(\bar s_{i,j}, \phi(c_{i,j}))$.
Recall that $\pi:V\rightarrow \Gamma$ sends each $v\in V$ (which is
a path in $\mathcal{K}$) to its terminal point.

\begin{lemma}\label{properties of rational constraints}
Suppose $\{\tilde v_{i, j}, \tilde p_{i,j}\}$ is a solution of $\mathcal{V}_{t}$. The following are true:
\begin{enumerate}

\item $\sigma_{q}\big(\pi(\tilde p_{i,j}),\pi(c_{i,j})\big)=a_{i,j}$;
\item $\sigma_{q}\big(\pi(\tilde p_{i,j}c_{i,j}), \pi((\tilde p_{i,j+1})^{-1})\big)=b_{i,j}$;
\item $Pa\Big(\sigma_{\rho}\big(\pi(\tilde v_{i,j}),\pi((\tilde v_{i,j})^{-1})\big)\Big)=d_{i,j}$.
\item $\pi(\tilde v_{i,j})=p(e_{i,j})$ if $e_{i,j}$ is a constant.
\end{enumerate}
\end{lemma}

\begin{proof}
First note for $v_{i}\in V$ and $w_{i}\in Y^{*}$ represents $v_{i}$ for $i=1,2$. we have
\begin{equation*}
 \sigma_{q}\big(\pi(v_{1}),\pi(v_{2})\big)=\sigma_{q}\big(\phi(w_{1}),\phi(w_{2})\big)
\end{equation*}
by the following commutative diagram
\[
\begin{CD}
Y^{*} @> \phi >> X^{*} \\
@V VV @V VV \\
V @> \pi >> \Gamma
\end{CD}
\]
Hence we can prove any statement about $\sigma_{q}\big(\pi(v_{1}),\pi(v_{2})\big)$ by proving the same statement about $\sigma_{q}\big(\phi(w_{1}),\phi(w_{2})\big)$

Since $\tilde p_{i,j}\in \phi^{-1}(L(\bar s_{i,j}))$,
there exists $p'_{i,j}\in Y^{*}$ representing $\tilde  p_{i,j}$ such that $\phi(p'_{i,j})\in L(\bar s_{i,j})$. Hence by Lemma \ref{key property of FPA} and the
definition of $a_{i,j}$, we have
$\sigma_{q}\big(\phi(p'_{i,j}),\phi(c_{i,j})\big)=a_{i,j}$, which proves
(1).

For (2), since $(\tilde p_{i,j+1})^{-1}\in \phi^{-1}(L(b_{i,j}))$, there exists $p'_{i,j+1}\in Y^{*}$ representing $\tilde p_{i,j+1}$ such that $\phi\big((p'_{i,j+1})^{-1}\big)\in L(b_{i,j})$. By the definition of $L(b_{i,j})$ we know that $\sigma_{q}\big(u_{i,j},
\phi((p'_{i,j+1})^{-1})\big)=b_{i,j}$ and that $\phi\big((p'_{i,j+1})^{-1}\big)$
is compatible with $\bar s'_{i,j}$. We know that
$\phi(p'_{i,j}c_{i,j})\in L(\bar s'_{i,j})$ by the definition of
$\bar s'_{i,j}$. We have $u_{i,j}\in L(\bar s'_{i,j})$. Hence by Lemma \ref{key property of FPA}
we have
\begin{equation*}
\sigma_{q}\big(\phi(p'_{i,j}c_{i,j}), \phi((p'_{i,j+1})^{-1})\big)=\sigma_{q}\big(u_{i,j},
\phi((p'_{i,j+1})^{-1})\big)=b_{i,j}
\end{equation*}

For (3), since $\tilde v_{i,j}\in\phi^{-1}(L(d_{i,j}))$, there exists
$v'_{i,j}\in Y^{*}$ representing $\tilde v_{i,j}$ such that
$\phi(v'_{i,j})\in L(d_{i,j})$. Hence by Lemma \ref{key property of PPA} we have
$Pa\Big(\sigma_{\rho}\big(\phi(v'_{i,j}),
(\phi(v'_{i,j}))^{-1}\big)\Big)=d_{i,j}$.

(4) directly follows from the definition of $L(e_{i,j})$.
\end{proof}

We now define the equation system $\mathcal{W}_{t}$ in $A$ correspond to $\mathcal{V}_{t}$.

Let $\{w_{i,j}\mid 1\leq i\leq n, 1\leq j\leq 3 \}$ be a set of constants and variables satisfying the follows:
\begin{enumerate}
\item $w_{i,j}=w_{i',j'}$ if and only if $e_{i,j}=e_{i',j'}$.
\item If $e_{i,j}$ is a variable in $\mathcal{E}$, then $w_{i,j}$ is a variable in $\mathcal{W}_{t}$.
\item When $e_{i,j}$ is a constant in $E$, we define
\begin{equation*}
w_{i,j}=\iota_{1}^{-1}\big(\iota_{2}(e_{i,j})\cdot (qp(e_{i,j}))^{-1}\cdot\iota_{4}(d_{i,j})\big)
\end{equation*}
\end{enumerate}

Note that in the last case $w_{i,j}$ may not be well defined since $\iota_{1}$ is not surjective. If this happens, we define
$\mathcal{W}_{t}$ to have no solution.

The equation system
$\mathcal{W}_{t}$ over $A$ is defined as follow:
\begin{equation*}
\mathcal{W}_{t}= \left\{ \begin{array}{rl}
\sum_{j=1}^{3}w_{i,j}=\iota_{1}^{-1}\Big(\sum_{j=1}^{3}(a_{i,j}+b_{i,j}+\iota_{4}(d_{i,j}))-\sigma_{q}\big(\pi(c_{i,1}),
\pi(c_{i, 2})\big)\Big);
\end{array} \right.  1\leq i\leq n
\end{equation*}
Note that the right hand sides of the equations above might not be
well defined since $\iota_{1}$ is not surjective. In that case, we
define $\mathcal{W}_{t}$ to have no solution.

\begin{theorem}\label{E V W}
$\mathcal E$ has a solution in $E$ if and only if
$\mathcal{E}_{t}=\mathcal{V}_{t}\cup\mathcal{W}_{t}$ constructed
above has a solution for some $t\in\Theta$.
\end{theorem}

\begin{proof}
Suppose $\mathcal{E}_{t}$ has a solution for $t=\big((c_{i,j}),
(\bar s_{i,j}), (b_{i,j}), (d_{i,j})\big)$, i.e. $\mathcal{V}_{t}$ has
a solution in $V$ and $\mathcal{W}_{t}$ has a solution in $A$. Let
$\{\tilde v_{i, j}, \tilde p_{i,j}\}$ be a solution of $\mathcal{V}_{t}$
and $\{\tilde w_{i,j}\}$ be a solution of $\mathcal{W}_{t}$.

Recall that $i$ is the inclusion from $A'$ to $E'$, $\iota_{1}$ is embedding of $A$ into $A'$ and $\iota_{2}$ is
the embedding of $E$ into $E'$.
We will
show that
\begin{equation*}
\tilde e_{i,j}=q\big(\pi(\tilde v_{i, j})\big)i\big(\iota_{1}(\tilde
w_{i,j})-\iota_{4}(d_{i,j})\big)
\end{equation*}
is a solution of $\mathcal{E}$ in $E'$. Here we think of $\mathcal{E}$ as an equation system in $E'$
by replacing all constants by their image under $\iota_{2}$.

To simplify notion let ${\bf v}_{i,j}=\pi(\tilde v_{i,j})$, ${\bf
p}_{i,j}=\pi(\tilde p_{i,j})$ and ${\bf c}_{i,j}=\pi(c_{i,j})$.

First note that $\tilde e_{i,j}$ has $q$-coordinates $\big({\bf
v}_{i,j}, \iota_{1}(\tilde w_{i,j})-\iota_{4}(d_{i,j})\big)$. By
direct computation, we have:

\begin{eqnarray*}
& &\tilde e_{i,1}\tilde e_{i,2}\tilde e_{i,3}
\\&=& \big({\bf v}_{i,1}, \iota_{1}(\tilde w_{i,1})-\iota_{4}(d_{i,1})\big)
          \big({\bf v}_{i,2}, \iota_{1}(\tilde w_{i,2})-\iota_{4}(d_{i,2})\big)
          \big({\bf v}_{i,3}, \iota_{1}(\tilde w_{i,3})-\iota_{4}(d_{i,3})\big)
\\&=& \big({\bf v}_{i, 1}{\bf v}_{i, 2}{\bf v}_{i, 3}, \hspace{1mm} \sigma_{q}({\bf v}_{i, 1}, {\bf v}_{i, 2})+\sigma_{q}({\bf v}_{i, 1}{\bf v}_{i, 2},{\bf v}_{i, 3})+ \sum_{j=1}^{3}(\iota_{1}(\tilde w_{i,j})-\iota_{4}(d_{i,j}))\big )\hspace{8mm}(1)
\end{eqnarray*}

By the definition of $\mathcal{V}_{t}$  and the fact that $\{\tilde v_{i,j}\}$ is a solution of $\mathcal{V}_{t}$, we have ${\bf v}_{i, 1}{\bf v}_{i, 2}{\bf v}_{i, 3}=1$ in $\Gamma$ .

Now we consider the second
component of $(1)$. The following claim reduces the first two terms into something we have control over.

\begin{Claim}
\begin{eqnarray*}
&&\sigma_{q}({\bf v}_{i, 1}, {\bf v}_{i, 2})+\sigma_{q}({\bf v}_{i, 1}{\bf v}_{i, 2},{\bf v}_{i, 3})
\\&=&\sigma_{q}({\bf c}_{i,1}, {\bf c}_{i,2})-\sum_{j=1}^{3}\sigma_{q}({\bf p}_{i,j}, {\bf c}_{i,j})-\sum_{j=1}^{3}\sigma_{q}({\bf p}_{i,j}{\bf c}_{i,j}, {\bf p}_{i,j+1}^{-1})
\end{eqnarray*}
\end{Claim}

\begin{proof}
Since $\{\tilde v_{i, j}, \tilde p_{i,j}\}$ is a solution of $\mathcal{V}_{t}$.  We have
\begin{equation*}
\left\{ \begin{array}{rl} \tilde p_{i,1}c_{i,1}(\tilde p_{i,2})^{-1}=\tilde v_{i,1}
                                \\ \tilde p_{i,2}c_{i,2}(\tilde p_{i,3})^{-1}=\tilde v_{i,2}
                                \\ \tilde p_{i,3}c_{i,3}(\tilde p_{i,1})^{-1}=\tilde v_{i,3}

                         \end{array} \right.   \hspace{4mm}   1\leq i\leq n\end{equation*}
Project these equations to $\Gamma$ by $\pi$, we have
\begin{equation*}
\left\{ \begin{array}{rl} {\bf p}_{i,1}{\bf c}_{i,1}({\bf p}_{i,2})^{-1}={\bf v}_{i,1}
                                \\  {\bf p}_{i,2}{\bf c}_{i,2}({\bf p}_{i,3})^{-1}={\bf v}_{i,2}
                                \\  {\bf p}_{i,3}{\bf c}_{i,3}({\bf p}_{i,1})^{-1}={\bf v}_{i,3}
                         \end{array} \right.   \hspace{4mm}   1\leq i\leq n\end{equation*}
A direct computation using the cocycle condition of $\sigma_{q}$ and
the fact that $q$ is symmetric gives the identity in the claim.
\end{proof}

By (1) and (2) of Lemma \ref{properties of rational constraints}, we have
\begin{equation*}
\sum_{j=1}^{3}\sigma_{q}({\bf p}_{i,j}, {\bf c}_{i,j})=\sum_{j=1}^{3}a_{i,j}
\end{equation*}
and
\begin{equation*}
\sum_{j=1}^{3}\sigma_{q}({\bf p}_{i,j}{\bf c}_{i,j}, {\bf p}_{i,j+1}^{-1})=\sum_{j=1}^{3}b_{i,j}.
\end{equation*}

Now Claim 2 and the fact that $\{\tilde w_{i,j}\}$ is a
solution of $\mathcal W_{t}$ tell us that the second component of $(1)$ equals
\begin{eqnarray*}
& & \sigma_{q}\big(\pi(c_{i,1}), \pi(c_{i,
2})\big)-\sum_{j=1}^{3}(a_{i,j}+b_{i,j})+\sum_{j=1}^{3}\iota_{1}(\tilde
w_{i,j})-\sum_{j=1}^{3}\iota_{4}(d_{i,j})
\\&=& \sigma_{q}\big(\pi(c_{i,1}), \pi(c_{i, 2})\big)-\sum_{j=1}^{3}(a_{i,j}+b_{i,j}+\iota_{4}(d_{i,j}))+\sum_{j=1}^{3}\iota_{1}(\tilde w_{i,j})=0 \hspace{4mm} (*)
\end{eqnarray*}

At this point, we have shown that $\tilde e_{i,1}\tilde e_{i,2}\tilde
e_{i,3}=1$ in $E'$.

\begin{Claim}
$\tilde e_{i,j}=q\big(\pi(\tilde v_{i, j})\big)i\big(\iota_{1}(\tilde
w_{i,j})-\iota_{4}(d_{i,j})\big)$ is in $\iota_{2}(E)$.
\end{Claim}

\begin{proof}
We use the $\rho'$-coordinate for $E'$. Note that an element $(g, a)_{\rho'}\in E'$ is in $\iota_{2}(E)$ if and
only if $a\in\iota_{1}(A)$ . By the definition of the symmetric section $q$, we have
\begin{eqnarray*}
\tilde e_{i,j}=\Big( \pi(\tilde v_{i,j}), -\iota_{3}\big(\sigma_{\rho}\big(\pi(\tilde v_{i, j}), \pi(\tilde
v^{-1}_{i, j})\big)\big)+\iota_{1}(\tilde
w_{i,j})-\iota_{4}(d_{i,j})\Big)_{\rho'}.
\end{eqnarray*}
Hence it is enough to show that $-\iota_{3}\big(\sigma_{\rho}\big(\pi(\tilde v_{i, j}), \pi(\tilde
v^{-1}_{i, j})\big)\big)-\iota_{4}(d_{i,j})\in\iota_{1}(A)$.
But this follows Lemma \ref{fixing the one half} since we know that
\begin{equation*}
Pa\Big(\sigma_{\rho}\big(\pi(\tilde v_{i,j}), (\pi(\tilde v^{-1}_{i, j})\big)\Big)=d_{i,j}.
\end{equation*}
by (3) of Lemma \ref{properties of rational constraints}. The proof
of the claim is complete.
\end{proof}

Therefore we know that $\mathcal{E}$ has a solution in $\iota_{2}(E)$.
Note that $\iota_{2}(E)$ is isomorphic to $E$. So $\mathcal{E}$ has
a solution in $E$. We have completed the proof of the ``if'' part of
the theorem at this point.

\vspace{4mm}

Now suppose $\mathcal{E}$ has a solution in $E$. Then $\mathcal{E}$ (with constants replaced by their images under $\iota_{2}$) has a solution in $\iota_{2}(E)\subset E'$. Let $\{\tilde e_{i,j}\}$ be such a
solution. We will show that one of the $\mathcal{E}_{t}$ we
constructed also has a solution.

First note that $\{p(\tilde e_{i,j})\}$ is a solution of $\mathcal{E}$
(with the constants replaced by their $p$ images) in $\Gamma$ . Then
by Proposition \ref{Dahmani-Guirardel} for  some $(c_{i,j})$
satisfying Condition 1, the tripod equation system
\begin{equation*}  \mathcal{V}_{t}^{1}=
\left\{ \begin{array}{rl} p_{i,1}c_{i,1}(p_{i,2})^{-1}=v_{i,1}
                                \\ p_{i,2}c_{i,2}(p_{i,3})^{-1}=v_{i,2}
                                \\ p_{i,3}c_{i,3}(p_{i,1})^{-1}=v_{i,3}
                               \end{array} \right.   \hspace{4mm}   1\leq i\leq n\end{equation*}
has a solution $\{\tilde p_{i,j},\tilde v_{i,j}\}$ in $V$ such that $\pi(\tilde v_{i,j})=p(\tilde
e_{i,j})$.

By Proposition \ref{Dahmani-Guirardel} we know that $\tilde p_{i,j},\tilde v_{i,j}$ are
$(\lambda_{1}, \nu_{1})$-quasi geodesics in $\mathcal{K}$. Hence by
Lemma \ref{quasi-geodesic constants} there exists $p'_{i,j},v'_{i,j}\in Y^{*}$ representing
$\tilde p_{i,j},\tilde v_{i,j}$ such that $\phi(p'_{i,j})$,
$\phi(v'_{i,j})$ are $(\lambda, \nu)$-quasi geodesics in
$K_{\Gamma}$ and $p'_{i,j}$ is a subword of $v'_{i,j}$. Recall in the definition of the tripod equation system, $v_{i,j}$ and $v_{i', j'}$ are defined to be the same variable if $e_{i, j}$ and $e_{i', j'}$ are the same variable. However we don't require $v'_{i,j}$ and $v'_{i',j'}$ to be the same $Y$-word even if $v_{i,j}$ and $v_{i', j'}$ are the same variable (hence $\tilde v_{i,j}=\tilde v_{i',j'}$).

Use the Future Predicting Automaton $M$ to read $\phi(p'_{i,j})$. Suppose it ends at the state
$\bar s_{i,j}$.  Note that
$\bar s_{i,j}$ and $\phi(c_{i,j})$ are compatible since $\phi(v'_{i,j})$ is in
$L$ and $\phi(p'_{i,j})\phi(c_{i,j})$ is a subword of
$\phi(v'_{i,j})$. Hence $(\bar s_{i,j})$ satisfy Condition 2. With the above choice of $c_{i,j}$ and $\bar s_{i,j}$, let $\bar s'_{i,j}$ be the state where $\phi(c_{i,j})$ ends when read by $M_{\bar s_{i,j}}$.

Note that $\phi(p'_{i,j}c_{i,j})\in L(\bar s'_{i,j})$. Also $\phi(p'^{-1}_{i,j+1})\in L(\bar s_{i,j}, c_{i,j})$ because
$\phi(p'_{i,j}c_{i,j}p'^{-1}_{i,j+1})$ is in $L$.
Let $b_{i,j}=\sigma_{q}\big(\pi(\tilde
p_{i,j}c_{i,j}), \pi(\tilde p^{-1}_{i,j+1})\big)$.
Then we have
$b_{i,j}\in A(\bar s_{i,j}, c_{i,j})=\{\sigma_{q}(\bar s'_{i,j},w)\mid w\in L(\bar s_{i,j}, c_{i,j}) \}$.
Therefore $(b_{i,j})$ satisfy Condition 3.

Use the Parity Predicting Automaton $D$ to read
$\phi(v'_{i,j})$. Let $d_{i,j}\in \mathbb{Z}_{2}^{n}\oplus\mathbb{Z}_{d_{1}}\oplus\cdots\oplus\mathbb{Z}_{d_{m}}$ be the last
component of the state where it ends.

The above choice of $\big((c_{i,j}), (\bar s_{i,j}), (b_{i,j}), (d_{i,j})\big)$
satisfies all conditions defining $\Theta$. Let $t=\big((c_{i,j}),
(\bar s_{i,j}), (b_{i,j}), (d_{i,j})\big)$.

Let $\mathcal{E}_{t}$ be the system of equations defined by the above $t$.
Then it is clear from the construction of $\mathcal{E}_{t}$ that $\{\tilde p_{i,j},\tilde v_{i,j}\}$ is a solution of it.

\vspace{4mm}
Let $\tilde e^{2}_{i,j}\in A'=\mathbb{Z}^{n}\oplus\mathbb{Z}_{2d_{1}}\oplus\cdots\oplus\mathbb{Z}_{2d_{m}} $ be the second component of the $\rho'$-coordinates of $\tilde e_{i,j}$. Since $\tilde e_{i,j}$ lies in $\iota_{2}(E)$, we know that $\tilde e^{2}_{i,j}\in \iota_{1}(A)$. By the above way of choosing $d_{i,j}$ and (3) of Lemma \ref{properties of rational constraints} we have
\begin{equation*}
d_{i,j}=Pa\Big(\sigma_{\rho}\big(\pi(\tilde v_{i,j}),\pi(\tilde
v^{-1}_{i,j})\big)\Big).
\end{equation*}
Therefore by Lemma \ref{fixing the one half} we know that
$\tilde e^{2}_{i,j}+\iota_{3}\Big(\sigma_{\rho}\big(\pi(\tilde
v_{i,j}),\pi(\tilde v^{-1}_{i,j})\big)\Big)+\iota_{4}(d_{i,j})\in A'$ lies in $\iota_{1}(A)$. Let
$\tilde w_{i,j}\in A$ be the unique element such that
\begin{equation*}
\iota_{1}(\tilde w_{i,j})=\bar e^{2}_{i,j}+\iota_{3}\Big(\sigma_{\rho}\big(\pi(\tilde
v_{i,j}),\pi(\tilde v^{-1}_{i,j})\big)\Big)+\iota_{4}(d_{i,j}).
\end{equation*}

\begin{Claim}
$\{\tilde  w_{i,j}\}$ is a solution of $\mathcal{W}_{t}$.
\end{Claim}

\begin{proof}
With all the notation above, we have
\begin{equation*}
\tilde e_{i,j}=q\big(\pi(\tilde v_{i,j})\big)i\big(\iota_{1}(\tilde w_{i,j})-d_{i,j}\big)
\end{equation*}
just as in the proof of the
``if'' part. But this time, we know that $\{\tilde e_{i,j}\}$ is a
solution of $E'$ instead of $\{\tilde w_{i,j}\}$ being a solution of
$\mathcal{W}_{t}$ and everything else is the same. So we can go
through the same calculation and when we reach $(*)$, we use the
fact that $\tilde e_{i,1}\tilde e_{i,2}\tilde e_{i,3}=1$ to conclude that
$(*)$ holds. Therefore $\{\tilde w_{i,j}\}$ is a solution of
$\mathcal{W}_{t}$ over $A$.
\end{proof}

The proof of the Theorem \ref{E V W} is complete.
\end{proof}

Theorem \ref{Main} follows from the Theorem \ref{E V W} since
equation systems with rational constraints in virtually free groups
are solvable by the work of Dahmani-Guirardel \cite{DGu} and
equation systems in finitely generated abelian groups can be solved
by using linear algebra.

\bibliography{temp}

\begin{thebibliography}{99}




\bibitem{Bridson-Haefliger}
M. R. Bridson and A. Haefliger.
\newblock {\em Metric spaces of non-positive curvature}, volume 319 of {\em
  Grundlehren der mathematischen {W}issenschaften}.
\newblock Springer, 1999.

\bibitem{BR}
M. R. Bridson and L. Reeves.
\newblock {\em On the algorithmic construction of classifying spaces and the isomorphism problem for biautomatic groups},
 \textit{Sci. China Math} {\bf 54} (2011), no. 8, 1533-1545.

\bibitem{DGu}
F. Dahmani and V. Guirardel, \newblock {\em Foliations for solving
equations in groups: free, virtually free and hyperbolic groups},
\textit{Journal of Topology} {\bf 3} (2010), 343-404.

\bibitem{DM}
V. Diekert and A. Muscholl, \newblock{\em Solvability of equations
in free partially commutative groups is decidable} In F. Orejas, P.
G. Spirakis, and J. van Leeuwen, editors, Proc. 28th International
Colloquium on Automata, Languages and Programming (ICALP'01), number
2076 in Lecture Notes in Computer Science, pages 543-554, Berlin
Heidelberg, 2001. Springer-Verlag.

\bibitem{Gromov} M. Gromov, Hyperbolic groups, {\em Essays in group theory}, \textit{Math. Sci.
Res. Inst. Publ.}, vol. 8, Springer, New York, (1987), 75--263.

\bibitem{Holt-Rees} D. F. Holt, S. Rees,
{\em Regularity of quasigeodesics in a hyperbolic group},
\textit{Internat. J. Algebra Comput}. {\bf 13} (2003), no. 5,
585每596.

\bibitem{LS}
M. Lohrey and G. Senizergues, \newblock{\em Theories of
HNN-extensions and amalgamated products} \textit{Proceedings
ICALP＊06, Part II, Lecture Notes in Computer Science} 4052
(Springer, Berlin, 2006) 504每515.

\bibitem{Mak}
G. S. Makanin, \newblock{\em Equations in a free group},
\textit{Izv. Akad. Nauk SSSR Ser. Mat.} {\bf 46} (1982)
1199-1273,1344.

\bibitem{NR1}
W. D. Neumann and L. Reeves,  \newblock {\em Regular cocycles and
biautomatic structures}, \textit{Internat. J. Algebra Comput.} {\bf
6} (1996), no. 3, 313-324.

\bibitem{NR2}
W. D. Neumann and L. Reeves, \newblock {\em Central extensions of
word hyperbolic groups}, \textit{Ann. of Math.} (2) {\bf 145}
(1997), no. 1, 183-192.

\bibitem{NS}
W. D. Neumann and M. Shapiro, \newblock {\em Equivalent automatic structures and their boundaries},
\textit{Internat. J. Algebra Comput.}  {\bf 2}  (1992), no. 4, 443-469.

\bibitem{RS}
E. Rips and Z. Sela, \newblock {\em Canonical representatives and
equations in hyperbolic groups}, \textit{Invent. Math.} {\bf 120}
(1998) 489-512.

\bibitem{Rom}
V. A. Roman'kov. \newblock {\em Universal theory of nilpotent
groups}, \textit{Mat. Zametki.}, {\bf 635} (1979) 25(4):487每495.

\bibitem{Serre}
J-P. Serre, \newblock{\em Trees}, \newblock Springer, Berlin (1980)
Translated from the French by John Stillwell

\bibitem{Tru}
J. K. Truss, \newblock {\em Equation-solving in free nilpotent
groups of class 2 and 3}, \textit{Bull. London Math. Soc.}, (1995)
27(1):39每45.


\end{thebibliography}

\end{document}